\documentclass[11pt]{amsart}
\usepackage{amssymb,amsfonts,graphicx,amsmath,amsthm,amstext,latexsym,amsfonts,bm}
\usepackage{graphicx,epsfig}
\usepackage{bbm}
\usepackage{t1enc}
\usepackage{mathrsfs}
\usepackage{subfig}
\usepackage{hyperref}
\usepackage{a4wide}
\usepackage{tikz}
\usetikzlibrary{intersections}
\usepackage{egothic}
\usepackage [T1] {fontenc}

\usepackage{mathtools}
\usepackage{esint}

\theoremstyle{plain}

\numberwithin{equation}{section}
\newtheorem{theorem}{Theorem}[section]
\newtheorem{corollary}[theorem]{Corollary}
\newtheorem{lemma}[theorem]{Lemma}
\newtheorem{definition}[theorem]{Definition}
\newtheorem{remark}[theorem]{Remark}

\newenvironment{proofad5}{\removelastskip\par\medskip
\noindent{\textbf {Proof of Corollary \ref{CR1}}.}
\rm}{\penalty-20\null\hfill$\blacksquare$\par\medbreak} 

\newenvironment{proofad6}{\removelastskip\par\medskip
\noindent{\textbf {Proof of Corollary \ref{CR2}}.}
\rm}{\penalty-20\null\hfill$\blacksquare$\par\medbreak} 

\usepackage{color}
\definecolor{darkred}{rgb}{0.8,0,0}
\definecolor{darkblue}{rgb}{0,0,0.7}
\definecolor{darkgreen}{rgb}{0,0.4,0}

\newcommand{\eps}{\varepsilon}

\newcommand{\R}{{\mathbb R}}

\newcommand{\W}{{\mathcal W}}

\newcommand{\V}{{\mathcal V}}

\newcommand{\un}{{\rm 1\kern -2.5pt l}}
\newcommand{\tr}{{\rm Tr}}

\def\w{\mathbf{w}}
\def\u{\mathbf{u}}

\def\yy{\mathbf{y}}

\def\eps{\varepsilon}
\def\R{{\mathbb R}}

\def\eps{\varepsilon}
\def\R{{\mathbb R}}

\def\F{{\mathcal F}}

\def\argmin{\mathop{{\rm argmin}}\nolimits}

\def\Tr{\mathop{{\rm Tr}}\nolimits}

\def\curl{\mathop{{\rm curl}}\nolimits}

\def\u{\mathbf{u}}
\def\v{\mathbf{v}}

\def\v{{\bf v}}
\def\w{{\bf w}}

\def\x{{\bf x}}

\def\wconv{\rightharpoonup}
\newcommand{\ZZZ}{\color{black}}
\newcommand{\MMM}{\color{black}}
\newcommand{\KKK}{\color{black}}

\renewcommand{\epsilon}{\varepsilon}
\newcommand{\beeq}{\begin{equation}}
\newcommand{\eneq}{\end{equation}}
\newcommand{\bear}{\begin{array}}
\newcommand{\enar}{\end{array}}
\newcommand{\bema}{\begin{displaymath}}
\newcommand{\enma}{\end{displaymath}}

\newcommand{\beea}{\begin{eqnarray}}
\newcommand{\enea}{\end{eqnarray}}

\newcommand{\om}{\Omega}

\newcommand{\bb}{\boldsymbol}

\newcommand{\lab}[1]{ \label{#1} }

\newenvironment{proofth1}{\removelastskip\par\medskip   
\noindent{\bf Proof of {\rm {\bf Theorem \ref{TH1}}}.}
\rm}{\penalty-20\null\hfill$\blacksquare$\par\medbreak} 

\def\wconv{\rightharpoonup}

\title[Asymptotic behavior of constrained local minimizers in finite elasticity]{Asymptotic behavior of constrained local minimizers in finite elasticity}
  \author{Edoardo Mainini}
\address[Edoardo Mainini]{Dipartimento di Ingegneria meccanica, energetica, gestionale e dei trasporti, 
  Universit\`a  degli studi di Genova, Via all'Opera Pia, 15 - 16145 Genova Italy.}
\email{mainini@dime.unige.it}

\author{Roberto Ognibene}
\address[Roberto Ognibene]{Dipartimento di Ingegneria meccanica, energetica, gestionale e dei trasporti, 
 Universit\`a  degli studi di Genova, Via all'Opera Pia, 15 - 16145 Genova Italy.}
\email{roberto.ognibene@edu.unige.it}

   \author{Danilo Percivale}
 \address[Danilo Percivale]{Dipartimento di Ingegneria meccanica, energetica, gestionale e dei trasporti, 
  Universit\`a  degli studi di Genova, Via all'Opera Pia, 15 - 16145 Genova Italy.}
  \email{percivale@dime.unige.it}
\begin{document}
 \maketitle
\begin{abstract} We provide an approximation result for the pure traction problem of linearized elasticity in terms of local minimizers of finite elasticity, under the constraint of  vanishing average curl for admissible deformation maps. When suitable rotations are included in the constraint, the limit  is shown to be the linear elastic equilibrium associated to rotated loads.
\end{abstract}


\section{Introduction}
We consider the pure traction problem of finite elasticity and its relation with the linear elastic problem.  If $\om\subset\R^3$ is the reference configuration of  an elastic body and $\mathbf y:\om\to \mathbb R^3$ is the deformation field,  we introduce the  global energy
\begin{equation}\label{un}
\displaystyle \mathcal G(\mathbf y;\mathcal L):=\int_\om\mathcal W(\x,\nabla\mathbf y)\,d\x
-\mathcal L(\mathbf y-\mathbf i),
\end{equation}
  where $\mathbf i$ denotes the identity map on $\om$  and $\mathcal W:\Omega\times\mathbb R^{3\times 3}\to[0,+\infty]$ is the strain energy density. Here, $\mathcal L$ is the load functional whose typical form is   
\begin{equation}\label{lod}\mathcal{L}(\v):=\int_\om \mathbf f\cdot \v\,d\x+\int_{\partial\om}\mathbf g\cdot \v\, dS,\end{equation}
 where $\mathbf{f}\colon\Omega\to \R^3$ is a body force field, $\mathbf{g}\colon\partial\Omega\to\R^3$ is a surface force field, $\mathbf{v}=\mathbf y-\mathbf i$ is the displacement field and 
  $dS$ is the surface measure.
 For  every $\x\in\om$, the function $\mathcal W(\x,\cdot)$ is assumed to be  frame indifferent, uniquely minimized at rotations with value $0$, $C^2$-smooth and quadratically growing out of rotations; in addition, $\mathcal W$ shall satisfy the natural condition
 \[ \mathcal W(\x,\mathbf F) = +\infty\ \text{ if } \  \det \mathbf F \leq 0.\]
 If $h>0$ is an adimensional parameter, we shall consider the natural rescaling $h^{-2}\mathcal G(\cdot;h\mathcal L)$.
 We also introduce the associated linear elastic energy
 \begin{equation*}
\mathcal F_0(\u;\mathcal L):= \frac{1}{2}\int_\Omega \mathbb E( \u)\, D^2\W(\x,\mathbf{I})\,\mathbb E( \u)\,d\x-\mathcal{L}(\u),
\end{equation*}
where  $\mathbb E(\u):=\mathrm{sym}(\nabla\u)$ is the infinitesimal strain tensor and $\mathbf I$ is the identity matrix.
This standard expression is formally obtained by linearization around the identity: by introducing the rescaled displacement $\mathbf u=h^{-1}\mathbf v$, by  writing the deformation field as $\mathbf i+h\mathbf u$ and by considering the rescaled energies  $h^{-2}\mathcal G(\mathbf i+h\mathbf u;h\mathcal L)$,  functional $\mathcal F_0(\mathbf u;\mathcal L)$ is the  pointwise limit obtained  by performing a Taylor expansion for small $h$.

  Recently, several contributions \cite{JS,  MPTJOTA, MPTARMA,  MPsharp, MM} have analyzed the variational limit of the rescaled functionals $h^{-2}\mathcal G(\cdot;h\mathcal L)$ as $h\to 0$. For this purpose, it is necessary to assume that $\mathcal L$ is equilibrated (i.e., with null resultant and null momentum) and that
   \beeq\label{L1plus}
 \mathcal L(\mathbf R\x-\x)\le 0 \qquad  \forall\, \mathbf R\in SO(3),
\eneq
   where $SO(3)$ denotes the special orthogonal group, thus preventing the rescaled energy being driven to $-\infty$ by rigid motions as $h$ becomes small.
 \KKK
  In contrast to the case of the Dirichlet problem \cite{ADMDS,DMPN,JS,MP,Sc}, global minimizers (or quasi-minimizers)  of $h^{-2}\mathcal G(\cdot;h\mathcal L)$ over $H^1(\om,\mathbb R^3)$ \KKK do not necessarily converge to minimizers of the associated   linearized elastic energy $\mathcal F_0(\cdot;\mathcal L)$,
  as might be expected from classical  continuum mechanics literature, so it seems  appropriate to shift the attention to the asymptotic behavior  of suitably constrained minimizers in order to recover a minimizer of $\mathcal F_0(\cdot;\mathcal L)$ in the limit. 
  Indeed, the results in \cite{JS, MPsharp, MM} show that, without further assumptions on the external loads, such as the absence of axes of equilibrium,
   there holds  $$\inf_{H^1(\om,\R^3)} h^{-2}\mathcal G(\cdot;h\mathcal L)\to \min_{H^1(\om,\R^3)} \mathcal F_0(\cdot;\mathbf R_*\mathcal L)\qquad\mbox{as $h\to 0$}$$
  for some suitable $\mathbf R_*\in SO(3)$, and  possibly $$ \min_{H^1(\om,\R^3)} \mathcal F_0(\cdot;\mathbf R_*\mathcal L)< \min_{H^1(\om,\R^3)} \mathcal F_0(\cdot;\mathcal L).$$\\

Before introducing constrained approximations,
let us mention that an alternative standard way for rigorously obtaining  linearized elasticity
is to consider the equilibrium problems associated to small loads. 
Equilibrium configurations of the (rescaled) traction problem for the equilibrated loads $(\mathbf f,\mathbf g)$ are given by the deformations $\yy\colon\Omega\to\R^3$ which solve 


\begin{equation}\label{eq:nonlinear_h}
	\left\{
	\begin{aligned}
		-\mathrm{div}\,(D\W(\x,\nabla\yy))&=h\mathbf{f}, &&\text{in }\Omega, \\
		D\W(\x,\nabla\yy)\mathbf{n}&=h\mathbf{g},&&\text{on }\partial\Omega,
	\end{aligned}
	\right.
\end{equation} where $\mathbf{n}$ denotes the exterior unit normal vector to $\Omega$.
Under suitable conditions on $\mathbf{f}$ and $\mathbf{g}$, classical results  based on  the implicit function theorem, described for instance in \cite{C,TN,V}, allow to prove that there exists  a solution $\mathbf y_h$ to the above problem
  in a small neighborhood of the identity and that  
$
\u_h:=h^{-1}({\yy_h-\mathbf{i}})\to 	\mathbf{u}_0
$
in a suitable Sobolev  space over $\Omega$, where $\u_0$ solves the linearized problem
\begin{equation*}
	\left\{\begin{aligned}
		-\mathrm{div}\,(D^2\W(\x,\mathbf{I})\mathbb E(\u_0))&=\mathbf{f}&&\text{in }\Omega, \\
		D^2\W(\x,\mathbf{I})\mathbb E(\u_0)\mathbf{n}&=\mathbf{g},&&\text{on }\partial\Omega,
	\end{aligned}\right.
\end{equation*}
that is, $\u_0$ is a minimizer of the linearized elastic energy $\mathcal F_0(\cdot;\mathcal L)$.
A drawback of  this scheme is that it works only if the external loads satisfy the following integrability condition
\beeq\lab{intforces}
\mathbf{f}\in L^p(\Omega,\R^3),\quad\mathbf{g}\in W^{1-1/p,p}(\partial\Omega,\R^3)\quad \text{for some }\  p> 3
\eneq
along with further nondegeneracy constraints (the simplest one being again the absence of axes of equilibrium) that we shall discuss in Section \ref{implicit} along with a precise statement and some literature review.\\

 
 \ZZZ
 In this paper, we would like to avoid the additional assumptions on external loads that are required both by the above classical approach and by the approximation with global minimizers of the rescaled functionals $h^{-2}\mathcal G(\cdot;h\mathcal L)$.
The approach we propose is to include a constraint on admissible deformation maps, without affecting the limit problem.
  \KKK
  The first step towards this goal is  a suitable definition of local minimizer of functionals \eqref{un}, which takes into account that the expected limit functional $\mathcal F_0(\cdot;\mathcal L)$ can be restricted, without loss of generality, to those  $\u\in H^1(\om,\mathbb R^3)$ such that \begin{equation}\label{curlu=0}\int_\om \curl \u=\mathbf 0,\end{equation} provided that $\mathcal L$ is equilibrated. Indeed, if $\u\in H^1(\om,\R^3)$ and $\w:=|\om|^{-1}\int_\om\mathrm{curl}\,\u$, it is immediate to check that $\mathcal F_{0}(\u-\tfrac12 \w\wedge\x;\mathcal L)=\mathcal F_{0}(\u;\mathcal L)$, where $\wedge$ denotes the cross product.  In the infinitesimal theory, the above constraint has a clear mechanical meaning: the infinitesimal rotation tensor $\mathrm{skew}(\nabla \u)$, whose axial vector is $\tfrac12\,\mathrm{curl}\,\u$, is vanishing in average.  \KKK
In this perspective we shall define constrained local minimizers $\mathbf y$ of $\mathcal G(\cdot;\mathcal L)$, the constraint being  \begin{equation}\label{curly=0}\int_\om \curl \mathbf y=\mathbf 0,\end{equation} 
 by requiring that $\mathcal G(\mathbf y;\mathcal L)\le \mathcal G(\mathbf y+\eps\boldsymbol\psi;\mathcal L)$ for every $\boldsymbol\psi$ such that $\int_\om \curl \boldsymbol\psi=\mathbf 0$ and for any small enough $\eps$ (we refer to Section \ref{mainsection} for the rigorous definition). 
In finite strain theory, 
the meaning of \eqref{curly=0} is obtained from the polar decomposition.
  Letting $\mathbf H:=|\Omega|^{-1}\int_\Omega\nabla\mathbf y$ be the average deformation gradient, and assuming that $ \mathbf H$ is positive definite, the polar decomposition of $\mathbf H$ provides a unique  $\mathbf R\in SO(3)$   such that
 $
 \mathbf H=\mathbf R\sqrt{\mathbf H^T\mathbf H}.
 $
 From the latter relation, and since $\mathbf H$ is positive definite, we deduce that  $\mathbf R=\mathbf I$ if and only if $\mathbf H^T=\mathbf H$, i.e., $\int_\Omega\mathrm{skew}(\nabla\mathbf y)=\mathbf 0$. Therefore,  \eqref{curly=0} means that the rotation tensor associated with the average deformation gradient is the identity. 
 Of course,  \eqref{curly=0} is equivalent to \eqref{curlu=0} if $\mathbf u$ is the (rescaled) displacement associated to the deformation $\mathbf y$, i.e., $\mathbf y=\mathbf i+h\u$, and in this case $|\Omega|^{-1}\int_\Omega\nabla\mathbf y$ is positive definite for small $h$.
 
 We notice that nontrivial rigid rotations are not admissible under the constraint \eqref{curly=0}.  In fact,
  since we deal with energetically stable configurations (local minimizers), let us  mention that, in the stability theory of elastic equilibria, other constraints such as the {\it zero moment condition}  by Beatty  \cite{Be} (see also \cite{F1,F2}) have been introduced for preventing instabilities due to rigid rotations. The zero moment condition restricts  admissible deformations to those  $\mathbf y$ that satisfy
  $\int_\Omega\mathbf f\wedge\mathbf y+\int_{\partial\Omega}\mathbf g\wedge \mathbf y\,dS=\mathbf 0.$
   It is interesting to notice that in a simple problem of homogeneous normal boundary traction, the zero moment condition coincides with \eqref{curly=0}.
 Indeed, if $\mathbf f \equiv\mathbf 0$ and $\mathbf g=\lambda\mathbf n$, where $\lambda\in \mathbb R$ and $\mathbf n$ is the exterior unit normal to $\partial\Omega$, then    $\int_{\partial\Omega} \mathbf g\wedge \mathbf y\,dS=
 \lambda\int_\Omega\mathrm{curl}\,\mathbf y$. 

  \KKK
 According to the above definition of constrained local minimizers, in our main result (see Theorem \ref{TH1} in Section \ref{mainsection}), under additional assumptions on $\W$, we show that if external loads are equilibrated then there exist  constrained local minimizers $\mathbf y_h\in H^1(\om,\mathbb R^3)$ for $\mathcal G(\cdot;h \mathcal L)$   such that
 by letting $\u_h:=h^{-1}(\mathbf y_h-\mathbf i)$ there hold
  \[ \begin{array} {ll} &\mathbb E(\u_h)\wconv \mathbb E(\u_0)\ \hbox{ weakly in }\ L^2(\om,\R^{3\times 3})\qquad\mbox{and}\qquad
 h^{-2}\mathcal G(\mathbf y_{h};h\mathcal L)\to {\mathcal F}_0(\u_0;\mathcal L)
 \end{array}
 \]
 as $h\to 0$, where  $\u_0$ minimizes $\mathcal F_0(\cdot;\mathcal L)$ over $H^1(\om,\R^3)$. \ZZZ 
    \KKK\\ 
 
Finally, by strengthening the hypotheses on external loads, namely by also assuming 
 \eqref{L1plus},
 an almost immediate consequence of the main result is that, given $\mathbf R\in SO(3)$ belonging to the rotation kernel
\beeq\lab{rotker}
\mathcal{S}^0_{\mathcal{L}}:=\{\mathbf{R}\in SO(3)\colon \mathcal{L}((\mathbf{R}-\mathbf{I})\x)=0\},
\eneq
for any  small enough $h$ there exists a constrained local minimizer $\mathbf y_h$ for $\mathcal G(\cdot, h\mathcal L)$ (this time the constraint for admissible deformations $\mathbf y$ being $ \int_\om \curl \mathbf R^T\mathbf y=\mathbf 0$), and 
 \[
 \lim_{h\to\infty}h^{-2}\mathcal G(\mathbf y_h;h\mathcal L)= \min_{\in H^1(\om,\mathbb R^3)} {\mathcal F}_0(\cdot;\mathbf R^T \mathcal L).
 \]
 This, in a nutshell, is what might be considered the essence of the approximation of minimizers of linear elasticity with constrained local minimizers of finite elasticity: when condition \eqref{L1plus} is satisfied then close to any $\mathbf R\in \mathcal{S}^0_{\mathcal{L}}$
 there exists a sequence of constrained local minimizers of functionals \eqref{un} such that the corresponding energies converge to the energy of the linearly elastic problem where $\mathcal L$ is replaced by $\mathbf R^T\mathcal L$. This further clarifies 
  that global minimizers of \eqref{un} are not a good choice in order to approximate minimizers of the linear elastic energy $\mathcal F_0(\cdot;\mathcal L)$.

 \subsection*{Plan of the paper} In section \ref{mainsection} we introduce all the assumptions of the theory and rigorously state the main results, which are proved in Section \ref{proofs}. In section \ref{implicit} we give a brief comparison between our approach and the classical results about the asymptotic behavior of equilibrium states of traction problems via implicit function theorem. In Section \ref{sectionexample} we revisit some examples of \cite{MPsharp} by applying our convergence results of constrained local minimizers.

\section{Main results}\label{mainsection}

In this section we introduce the basic notations and assumptions, then we state the main results.
In the following, $\Omega$ is a bounded open connected Lipschitz subset of $\mathbb R^3$, representing the reference configuration of the body.
$\mathbb R^{3\times3}$ is the set of real $3\times3$ matrices. $\mathbb R^{3\times3}_{\mathrm{sym}}$ (resp. $\mathbb R^{3\times3}_{\mathrm{skew}}$) is the set of symmetric (resp. skew-symmetric) matrices. $\mathbb R^{3\times3}_+$ denotes the set of matrices with positive determinant. 

 \subsection*{Assumptions on the elastic energy density} 
We let
  $\mathcal W : \om \times \mathbb R^{3 \times 3} \to [0, +\infty ]$ be  
 a Carath\'eodory function. \KKK We will consider the following assumptions. 
%
%
\beeq \lab{framind}\tag{$\bb{\mathcal W1}$} \W(\x, \mathbf R\mathbf F)=\W(\x, \mathbf F) \qquad \forall \,  \mathbf R \in SO(3) \quad \forall\, \mathbf F\in \mathbb R^{3 \times 3},\qquad \mbox{for a.e. $\x\in\Omega$},
\eneq
\beeq \lab{Z1}\tag{$\bb{\mathcal W2}$}
\min \mathcal W=\KKK	\W(\x,\mathbf I)=0 \quad \mbox{for a.e. $\x\in\Omega$}.
\eneq
Moreover, we shall consider the following regularity property:  there exist an open neighborhood $\mathcal U$ of $SO(3)$ in $\R^{3\times3}$,  an increasing function $\omega:\mathbb R_+\to\mathbb R$ satisfying $\lim_{t\to0^+}\omega(t)=0$ and a constant $K>0$
such that for a.e. $\x\in\om$
\beeq\begin{array}{ll}\lab{reg}\tag{$\bb{\mathcal W3}$} &   
\vspace{0,1cm}
\mathcal W(\x,\cdot)\in C^{2}(\mathcal U),\;\;\;
 \vspace{0,1cm}
  |D^2 \mathcal W(\x,\mathbf I)|\le K \;\;\hbox{and}\\
& 
 |D^2\W(\x,\mathbf F)-D^2\W(\x,\mathbf G)|\le\omega(|\mathbf F-\mathbf G|)\quad\forall\; \mathbf F,\mathbf G\in\mathcal U.
\end{array}
\eneq
We introduce a growth conditions from below: there exist $C>0$ such that for a.e. $\x\in\Omega$
\beeq \lab{coerc}\tag{$\bb{\mathcal W4}$}
\begin{array}{ll}
\W(\x,\mathbf F)\ge  C\, d(\mathbf F, SO(3))^2\qquad
 \forall\, \mathbf F\in \mathbb R^{3 \times 3},
\end{array}
\eneq
where $\mathrm{dist}(\mathbf F, SO(3)):=\inf\{|\mathbf F-\mathbf R|: \mathbf R\in SO(3)\}$ and $|\mathbf F|^2:=\mathrm{Tr}(\mathbf F^T\mathbf F)$.
 A second growth condition from below that we shall consider  is the following: there exist $ C'>0$,  $s\ge 2$, $q\ge \tfrac{s}{s-1}$ and $r>1$ such that for a.e. $\x\in\om$
 \beeq \lab{coerc2}\tag{$\bb{\mathcal W5}$}
 \begin{array}{ll} 
 	& \W(\x,\mathbf F)\ge  C' (|\mathbf F|^s+|\mathrm{cof}\,\mathbf F|^q+(\det\mathbf F)^r-1)\qquad
 	\forall\, \mathbf F\in \mathbb R^{3 \times 3}\end{array}
 \eneq
Finally, we introduce a polyconvexity condition: for a.e. $\x\in\om$
 \beeq\begin{array}{ll}\lab{poly}\tag{$\bb{\mathcal W6}$} &   
\hbox{the map $\mathbb R^{3\times3} \KKK\ni \mathbf F\mapsto\mathcal W(\x,\mathbf F)$ is polyconvex and}
\vspace{0.21cm}
\\&\displaystyle\mathcal W(\x,\mathbf F)=+\infty\text{ if }\det\mathbf F\le 0,\quad \lim_{\det\mathbf F\to 0^+}\mathcal W(\x,\mathbf F)=+\infty.\\
\end{array}
\eneq
\subsection*{Model energy densities}

We present here two instances of energy densities $\W$ which satisfy the above assumptions  and for which the main result of the present paper (see Theorem \ref{TH1} below) applies. \MMM For simplicity, we consider the homogeneous case. \KKK Incompressible materials are usually modeled by isochoric-type energies $\mathcal{W}_{\textup{iso}}$ defined on the set of matrices with unitary determinant $\{\mathbf{F}\in\R^{3\times 3}\colon \det \mathbf{F}=1\}$ and one can pass to the corresponding compressible model by letting \MMM
\begin{equation}\label{iso-vol}
\mathcal{W}(\mathbf{F}):=\left\{\begin{array}{ll}\displaystyle\mathcal{W}_{\textup{iso}}\left(\frac{\mathbf{F}}{(\det\mathbf{F})^{1/3}}\right)+\mathcal{W}_{\textup{vol}}(\mathbf{F})\qquad&\mbox{if $\det\mathbf F>0$}\\
+\infty\qquad&\mbox{if $\det\mathbf F\le 0$},
\end{array}\right.
\end{equation}\KKK
 where $\mathcal{W}_{\textup{vol}}(\mathbf{F})=g(\det\mathbf{F})$ for some convex $g\colon \R_+\to \R$  of class $C^2$ in a neighborhood of $1$ and such that
\begin{equation}\label{guno}
g(t)\geq 0\text{ for all }t>0,\quad g(t)=0 \text{ if and only if }t=1,\quad g''(1)>0,\quad \lim_{t\to 0^+}g(t)=+\infty.
\end{equation}
In addition, the function $g$ is  required to grow faster than linearly at infinity, i.e.
\begin{equation}\label{phi_growth}
g(t)	\geq C'' t^r,\quad\text{for }t>0\text{ sufficiently large and for some }C''>0\text{ and } r>1.
\end{equation}
In view of the isochoric-volumetric decomposition \eqref{iso-vol}, a model energy density is identified by the choice of $\W_{\textup{iso}}$, while $\W_{\textup{vol}}$ is left in a general form satisfying the above restrictions. 

The first example is an energy of Yeoh type, which is defined by choosing
\begin{equation}\label{Yeoh}
\W_{\textup{iso}}(\mathbf{F}):=\sum_{k=1}^3 c_k(|\mathbf{F}|^2-3)^k
\end{equation}
with coefficients $c_k> 0$. It is easy to check that with this choice the  energy density  satisfies all the assumptions from \eqref{framind} to \eqref{poly}, \MMM provided that $r\ge 2$ in \eqref{phi_growth}. \KKK Indeed, the validity of \eqref{framind}, \eqref{Z1}, \eqref{reg} is trivial. \MMM The validity of $\eqref{poly}$ follows from the polyconvexity of the map $\mathbb R^{3\times3}_+\ni\mathbf F\mapsto\frac{|\mathbf F|^2}{(\det\mathbf F)^{2/3}}$.  
The inequality in \eqref{coerc} is satisfied if $\mathrm{dist}(\mathbf F,SO(3))$ is small enough, as shown in \cite[Remark 2.8]{ADMDS}, and therefore \eqref{coerc}  directly follows from the following claim: there are positive constants $a_1,a_2$ such that $\mathcal W(\mathbf F)\ge a_1|\mathbf F|^3-a_2$ for every $\mathbf F\in\mathbb R^{3\times3}$. In order to prove such a claim, it is clear that  we can reduce to consider the regime $\mathrm{dist}(\mathbf F,SO(3))\gg1$ (i.e., $|\mathbf F|\gg1$). 
   By \eqref{phi_growth} there is $t_0>1$ such that $g(t)\ge C'' t^r$ for $t\ge t_0$ and the claim is obvious for $\det \mathbf F<t_0$ so that we may assume $\det\mathbf F\ge t_0$. If $|\mathbf F|\le (\det\mathbf F)^{2/3}$  we have $g(\det\mathbf F)\ge C''(\det\mathbf F)^r\ge C''|\mathbf F|^{3r/2}$ and the claim follows since $r\ge 2$; else if $|\mathbf F|\ge (\det\mathbf F)^{2/3}$ then $|\mathbf F|^2/(\det\mathbf F)^{2/3}$ is large and  Young inequality entails the existence of suitable positive constants $c_3', c_3''$ such that
\[
\mathcal W(\mathbf F)\ge c_3\left(\frac{|\mathbf F|^2}{(\det\mathbf F)^{2/3}}-3\right)^3+g(\det\mathbf F)\ge c_3'\left(\frac{|\mathbf F|^6}{(\det\mathbf F)^{2}}+(\det\mathbf F)^r\right)\ge c_3''\,|\mathbf F|^{\frac{6r}{2+r}}.
\]    
The claim is proved since $r\ge 2$, and it implies \eqref{coerc2} by means of the elementary inequality $|\mathrm{cof}\,\mathbf F|\le 2|\mathbf F|^2$.
\KKK

A second example of energy density for hyperelastic materials is given by  Ogden type energies, identified by the following choice
\[
\W_{\textup{iso}}(\mathbf{F}):=\sum_{i=1}^n c_i \left(\Tr(\mathbf{F}^T\mathbf{F})^{\gamma_i/2}-3\right)+\sum_{j=1}^m d_k\left(\Tr(\mathrm{cof}\mathbf{F}^T\mathrm{cof}\mathbf{F})^{\delta_j/2}-3\right),
\]
where $n,m$ are nonnegative integers and $c_i,\gamma_i,d_j,\delta_j\geq 0$. While hypotheses \eqref{framind},\eqref{Z1} and \eqref{reg} are again easily seen to be fulfilled by this model, some care must be taken concerning the remaining assumptions. In particular, one can prove that \eqref{coerc2} and \eqref{poly} hold true, after suitably restricting the values of the exponents $\gamma_i$ and $\delta_j$. More precisely, in order to obtain polyconvexity of the function $\mathcal{W}$ (i.e. assumption \eqref{poly}), one shall assume that
\[
	\gamma_i\geq \frac{3}{2}\quad\text{for all }i=1,\dots,n\quad\text{and}\quad \delta_j\geq 3\quad\text{for all }j=1,\dots,m.
\]
Moreover, if we let $\gamma:=\max\{\gamma_i\colon i=1,\dots,n\}$ and $\delta:=\max\{\delta_j\colon j=1,\dots,m\}$, we have that \eqref{coerc2} holds with $r>1$ as in \eqref{phi_growth} together with
\[
	s:=\frac{3\gamma r}{3r+\gamma}\quad\text{and}\quad q:=\frac{3\delta r}{2\delta+3r},
\]
provided $s\geq 2$ and $q\geq s/(s-1)$. We point out that $\delta\geq 3$ and $r>1$ imply $q>1$. For instance, \eqref{coerc2} is satisfied if \MMM $\gamma\geq 5/2$ and $r\ge 4$ 
\KKK The proof of \eqref{coerc2} and \eqref{poly} can be found in \cite[Proposition 6, Proposition 7]{CDHL}. Finally, for the proof of \eqref{coerc} one can see \cite[Remark 2.8]{ADMDS} and no further restrictions on $\gamma_i$ and $\delta_j$ are needed. 

\bigskip

\subsection*{Energy functionals}
Let  $\mathcal W$ satisfy assumptions \eqref{framind}, \eqref{Z1}, \eqref{reg} and \eqref{coerc}. We denote by $(H^1(\om,\mathbb R^3))^*$
the dual of the Sobolev space $H^1(\om,\mathbb R^3)$.
Given
 $ \mathcal L\in (H^1(\om,\mathbb R^3))^*$, we introduce the following energy functionals. We let $\mathcal G(\cdot;\mathcal L):H^1(\om,\mathbb R^3)\to (-\infty,+\infty)$ be defined by
\begin{equation}\label{functional}
\mathcal G(\mathbf y;\mathcal L):=\int_\om\mathcal W(\x,\nabla\mathbf y(\x))\,d\x-\mathcal L(\mathbf y-\mathbf i),
\end{equation}
and for $\u\in H^1(\om,\mathbb R^3)$ we let
\begin{equation*}
\mathcal F(\u;\mathcal L):=\int_\om\mathcal W(\x,\mathbf I+\nabla\mathbf u(\x))\,d\x-\mathcal L(\u)=\mathcal G(\mathbf i+\u;\mathcal L).
\end{equation*}
We define $\mathcal F_0:H^1(\om,\mathbb R^3)\to (-\infty,+\infty)$ as 
\beeq\label{lineare}
\mathcal F_0(\u;\mathcal L):=\frac{1}{2}\int_\om \mathbb E(\u)\,D^2\mathcal W(x,\mathbf I)\,\mathbb E(\u)\,d\x-\mathcal L(\u)
\eneq
If $h> 0$ we also define the rescaled functionals 
\beeq\label{Fh}
\mathcal F_h(\u;\mathcal L):=h^{-2}\mathcal F(h\u; h\mathcal L).
\eneq
We  refer to $\mathcal L$ as the {\it load functional}. Thanks to the Sobolev embedding $H^1(\om,\mathbb R^3)\hookrightarrow L^6(\om,\mathbb R^3)$, it can be always written as  
$$\mathcal L(\u)=\int_\om (\mathbf f_*\cdot \u+\mathbf G_*: \nabla \u)\,d\x$$
for suitable $\mathbf f_*\in L^{\frac{6}{5}}(\om,\mathbb R^3),\ \mathbf G_*\in L^2(\om,\mathbb R^{3\times3})$. We  assume that $\mathcal L$ is  equilibrated, 
i.e., 
\begin{equation}\label{L1}
\mathcal L(\mathbf c)=\mathcal L(\mathbf W\x)=0\quad\mbox{ for every $\mathbf c\in\mathbb R^3$ and every $\mathbf W\in \mathbb R^{3\times3}_{\mathrm{skew}}$}.
\end{equation}
We denote with $\|\mathcal L\|_*$ its norm in the space $(H^1(\om,\mathbb R^3))^*$.
By \eqref{L1} and by Korn and Poincar\'e inequalities,  there exists a constant $K_\Omega>0$ such that
\begin{equation}\label{KP}
\mathcal L( \u)\le K_\Omega\|\mathcal L\|_{ *}\,\|\mathbb E\u\|_{L^2(\om,\mathbb R^{3\times3})},
\end{equation}
for every $\u\in H^1(\om,\mathbb R^3)$.
 
Concerning the elastic energy, the following rigidity inequality is crucial, see \cite{ADMDS,FJM0, FJM}: there exists a constant $\tilde C=\tilde C_{\om}$ such that for every $\mathbf y \in H^1(\om)$ there is $\mathbf R\in SO(3)$ (depending on $\mathbf y$) such that
\[
\int_\om |\nabla\mathbf y(\x)-\mathbf R|^2\,d\x\le \tilde C_{\om}\int_\om \mathrm{dist}(\nabla\mathbf y(\x), SO(3))^2\,d\x. 
\]
For every $\mathbf y\in H^1(\om,\mathbb R^3)$, we combine the latter with \eqref{coerc} and obtain the  existence of  $\mathbf R \in SO(3)$ (depending on $\mathbf y$) such that
\begin{equation}\label{rigidity}
\int_\om |\nabla\mathbf y(\x)-\mathbf R|^2\,d\x\le {C_{\om}} \int_\om \mathcal W(\x,\nabla\mathbf y(\x))\,d\x,
\end{equation}
where $C_{\om}:=\tilde C_{\om}/C$ and $C$ is the constant in \eqref{coerc}. 

\subsection*{Main results}
We state the main result, showing that there exist suitably constrained local minimizers of the rescaled functionals \eqref{Fh} that converge to a global minimizer of the linear elastic energy \eqref{lineare} as $h\to0$.  We need the precise notion local minimizer under the constraint of vanishing average curl.
In the next definition we make use of the following notation: for $\mathbf R\in SO(3)$, we let
\[
H^1_{\mathbf R}(\om,\mathbb R^3):=\left\{\u\in H^1(\om,\mathbb R^3):\, \int_\om\mathrm{curl}\,\mathbf R^T\u=\mathbf 0\right\}.
\]
In particular, $H^1_{\mathbf I}(\om,\mathbb R^3)$ is the linear subspace of $H^1(\om,\mathbb R^3)$ made of vector fields with vanishing average curl.
\begin{definition}[{\bf constrained local minimizer}]\label{constrained}
Let $\mathcal W$ satisfy assumptions  \eqref{framind}, \eqref{Z1}, \eqref{reg} and \eqref{coerc} and let $\mathbf R\in SO(3)$.
 We say that  $\mathbf y\in H^1_{\mathbf R}(\om,\mathbb R^3)$ is a \emph{local minimizer for  $\mathcal G(\cdot; \mathcal L)$ over $H^1_{\mathbf R}(\om,\mathbb R^3)$}
 if for every $\boldsymbol\psi \in H^1_{\mathbf R}(\om,\mathbb R^3)$ there exists
  $\eps_0=\eps_0(\mathbf y,\boldsymbol\psi)>0$ such that for every $\eps\in[0,\eps_0]$ there holds
  $\mathcal G(\mathbf y;\mathcal L)\le\mathcal G(\mathbf y+\eps\boldsymbol\psi; \mathcal L)$.
 We say that $\u\in H^1_{\mathbf R}(\om,\mathbb R^3)$  is a \emph{local minimizer for $\mathcal F(\cdot; \mathcal L)$  over $H^1_{\mathbf R}(\om,\mathbb R^3)$} if   $\mathbf y(\x)=\mathbf R\x+\u$ is a local minimizer for $\mathcal G(\cdot;\mathcal L)$ over $H^1_{\mathbf R}(\om,\mathbb R^3)$.
 \end{definition}


\begin{remark}\rm If $\mathbf y(\x)=\mathbf R\x+\u$ we have $\mathrm{curl}\, \mathbf R^T\mathbf y=\mathrm{curl}\,\mathbf R^T\u$. In the distinguished case $\mathbf R=\mathbf I$  we have $\mathrm{curl}\, \mathbf y=\mathrm{curl}\,\u$  and  in particular   $\mathbf i+\mathbf u\in H^1_{\mathbf I}(\om,\mathbb R^3)$ if and only if $\u\in H^1_{\mathbf I}(\om,\mathbb R^3).$ 
\end{remark}

\KKK

\begin{remark}\label{ftofh}\rm Given $h>0$, $\mathcal F_h$ is defined by \eqref{Fh}. By rescaling $\eps_0$ we notice that  $\u\in H^1_{\mathbf I}(\om,\mathbb R^3)$ is a local minimizer for $\mathcal F_h( 
\cdot;\mathcal L)$ over  $H^1_{\mathbf I}(\om,\mathbb R^3)$  if and only if $\v:=h\u$ is a local minimizer for $\mathcal F(\cdot;h\mathcal L)$ over $H^1_{\mathbf I}(\om,\mathbb R^3)$. \KKK
\end{remark}

\begin{remark}{\rm
By applying the classical results of \cite{B},  if $\mathcal W$ satisfies \eqref{framind}, \eqref{Z1}, \eqref{reg}, \eqref{coerc}, \eqref{coerc2}, \eqref{poly}  and if $\mathcal L\in (H^1(\om,\R^3))^*$ satisfies \eqref{L1}, then a global minimizer of $\mathcal F(\cdot;\mathcal L)$ over $H^1(\om,\R^3)$ does exist. Nevertheless, it has been recently  shown in \cite{MPsharp} that following a sequence of global minimizers for the rescaled functionals $\mathcal F_{h}(\cdot;\mathcal L)$, the limit energy as $h\to 0$ is not necessarily equal to $\min_{H^1(\om,\mathbb R^3)}\mathcal F_0(\cdot;\mathcal L)$. Indeed, in a pure traction problem, this limit  can be strictly lower than the
minimal  value of $\mathcal F_0(\cdot;\mathcal L)$ over $H^1(\om,\mathbb R^3)$.
 This motivates the analysis of constrained local minimizers for obtaining the linear elastic energy  as limit of rescaled finite elasticity energies.
}\end{remark}


\KKK

We are in a position to state our main result

\begin{theorem}\lab{TH1}
Suppose that \eqref{framind}, \eqref{Z1}, \eqref{reg}, \eqref{coerc}, \eqref{coerc2}, \eqref{poly} hold true.  Let  $\mathcal L \in (H^1(\om,\mathbb R^3))^*$ satisfy \eqref{L1}.
Then there exist a vanishing sequence $(h_j)_{j\in\mathbb N}\subset(0,1)$ and a sequence $(\u_j)_{j\in\mathbb N}\subset H^1_{\mathbf I}(\om,\mathbb R^3)$ such that $\u_j$ is a local minimizer for $\mathcal F_{h_j}(\cdot;\mathcal L)$ over $H^1_{\mathbf I}(\om,\mathbb R^3)  $ for any $j\in\mathbb N$ according to {\rm Definition \ref{constrained}}, and moreover
 \[ \mathbb E(\u_j)\wconv \mathbb E(\u_*)\ \hbox{ weakly in }\ L^2(\om,\mathbb R^{3\times3})\ \hbox{as $j\to\infty$},
 \]
 \begin{equation*}
 \lim_{j\to\infty}\mathcal F_{h_j}(\u_j;\mathcal L)= {\mathcal F}_0(\u_*;\mathcal L),
 \end{equation*}
where  $\u_*$ is a minimizer for  $\mathcal F_0(\cdot;\mathcal L)$ over $H^1(\om,\mathbb R^3)$.
\end{theorem}\KKK

The above Theorem \ref{TH1} expresses the fact that the minimizer of the linear elastic energy $\mathcal F_0(\cdot;\mathcal L)$ (which is unique up to infinitesimal rigid displacements) gets approximated by a sequence of constrained local minimizers of rescaled finite elasticity functionals.
We stress that the vanishing average curl constraint disappears in the limit problem, since the stored linear elastic energy and the load functional are  invariant by infinitesimal rigid displacements, due to \eqref{L1}.
Since only local minimizers are involved, the statements requires \eqref{L1} and no further assumptions on the external loads. In particular, assumption
\eqref{L1plus}, 
appearing in \cite{MPsharp} for the analysis of the asymptotic behavior of global minimizers (or quasi-minimizers) of rescaled finite elasticity functionals, is not required. Similarly, assumptions about axes of equilibrium of external loads (see Section \ref{implicit}) are not required as well.


However, if besides \eqref{L1} we additionally assume \eqref{L1plus}, we can draw some interesting consequences of Theorem \ref{TH1}.
Under assumption \ref{L1plus},
 it is not difficult to see that  the \emph{rotation kernel} of $\mathcal{L}$, defined by \eqref{rotker},
is a subgroup of $SO(3)$, see \cite[Remark 2.2]{MP}. $\mathcal{S}^0_{\mathcal{L}}$ contains the identity matrix and it is possibly reduced to it.   Moreover, if $\mathbf R\in \mathcal{S}^0_{\mathcal{L}}$, then for every $\mathbf S\in SO(3)$ we get 
\[\mathbf R^T\mathcal L((\mathbf S-\mathbf I)\x)= \mathcal L((\mathbf R\mathbf S-\mathbf I)\x)-\mathcal L((\mathbf R-\mathbf I)\x) =\mathcal L((\mathbf R\mathbf S-\mathbf I)\x)\le 0,
\]
that is, $\mathbf R^T\mathcal L$ still satisfies \eqref{L1plus} (hence it also satisfies\eqref{L1}, see \cite[Remark 2.1]{MP}) for every $\mathbf R\in \mathcal{S}^0_{\mathcal{L}}$. Therefore the map $\beta: \mathcal{S}^0_{\mathcal{L}}\to \mathbb R$
\beeq\lab{betti}
\beta(\mathbf R):=  \min \{{\mathcal F}_0(\u;\mathbf R^T \mathcal L): \u\in H^1(\om,\mathbb R^3)\}
\eneq
is well defined and continuous.
We notice that the energy identity satisfied by solutions to the minimization problem \eqref{betti}
 implies that $\beta (\mathbf R)=-\tfrac12\mathcal B(\mathbf R^T\mathcal L,\mathbf R^T\mathcal L)$, where $\mathcal B(\cdot,\cdot)$ is the Betti form associated to  couples of equilibrated loads, defined by
\[\mathcal B(\mathcal L_1,\mathcal L_2):=\int_\om \mathbb E\u_1\,D^2\W(\x,\mathbf I)\,\mathbb E \u_2\,d\x.\]
 Here, $\u_i$ are solutions the the linear elastic problem with external loads $\mathcal L_i$, $i=1,2$.
 The degeneracy properties of the Betti form are relevant in the analysis of the nonlinear elastic equations for pure traction problems, see \cite{CMW2}. A nondegeneracy condition on $\beta$ also appears in the following results
which are
 straightforward consequences of Theorem \ref{TH1}.
\begin{corollary}\lab{CR1}Assume that \eqref{framind}, \eqref{Z1}, \eqref{reg}, \eqref{coerc}, \eqref{coerc2}, \eqref{poly} hold true and that  $\mathcal L \in (H^1(\om,\mathbb R^3))^*$ satisfies  \eqref{L1} and \eqref{L1plus}
Let $\mathbf R\in \mathcal S^0_{\mathcal L}$.
Then there exist a vanishing sequence $(h_j)_{j\in\mathbb N}\subset(0,1)$ and a sequence $(\mathbf y_j)_{j\in\mathbb N}\subset H^1_{\mathbf R}(\om,\mathbb R^3)\KKK$  such that $\mathbf y_j$ is a   local minimizer for $\mathcal G(\cdot, h_j\mathcal L)$ over $ H^1_{\mathbf R}(\om:\mathbb R^3)$ for every $j\in\mathbb N$ according to {\rm Definition \ref{constrained}} and 
 \[
 \lim_{j\to\infty}h_j^{-2}\mathcal G(\mathbf y_{j},h_j\mathcal L)= \beta(\mathbf R).
 \]
\end{corollary}
\begin{corollary}\lab{CR2}In the same assumptions of {\rm Corollary \ref{CR1}}, suppose  that $\mathcal{S}^0_{\mathcal{L}}\setminus \{\mathbf I\}\neq\emptyset$ and that $\beta$ from \eqref{betti} is not a constant map. Then for every $n\in \mathbb N$ there exist $\{\mathbf R_k:\, k=1,2...n\}\subset \mathcal S^0_{\mathcal L}$ and \MMM $h_0=h_0(n)> 0$ \KKK such that for every $h\in (0,h_0)$ the functional $\mathcal G(\cdot\,; h\mathcal L)$ has a  local minimizer $\mathbf y_k$ over $ H^1_{\mathbf R_k}(\om,\mathbb R^3)$
 for every $k=1,2...n$ (according to {\rm Definition \ref{constrained}}) and $\mathcal G(\mathbf y_k; h\mathcal L)\not = \mathcal G(\mathbf y_{k'}; h\mathcal L)$ if $k\not = k'$.
\end{corollary}

 The last corollary 
shows that  the nonlinear elastic energy might have several constrained local minimizers (each corresponding to a different constraint) at different energy levels. \KKK
The condition requiring that $\beta$ is not a constant map is verified for instance in the example from \cite[Theorem 2.7]{MPsharp}, which we will further discuss in Section \ref{sectionexample}.



\section{The classical approach via implicit function theorem}
\label{implicit}
 
In this section we briefly review some known results concerning existence and asymptotic behavior of equilibrium configurations of hyperelastic bodies, that make use of  a perturbative method which  relies on the implicit function theorem. This is done for the sake of a comparison with the assumptions of our result in Theorem \ref{TH1}. Indeed, the classical results that we are going to describe hereafter require stronger summability and nondegeneracy assumptions on external loads, as well as more regularity of $\mathcal W$ (in turn, they do not need polyconvexity).

Throughout this section, we assume $\Omega\subset\R^3$ to be of class $C^2$, and besides \eqref{framind} and \eqref{Z1} we shall consider the following additional regularity hypothesis on the strain energy density 
\begin{equation}\label{w7}\tag{$\bb{\mathcal W7}$}
\mathcal W\in C^3(\overline{\Omega}\times \mathcal{U}) \quad
	\text{for some }\mathcal{U}\subseteq\R^{3\times 3}\text{ open neighbourhood of }SO(3),
\end{equation}
along with the coercivity condition
\begin{equation}\label{w8}\tag{$\bb{\mathcal W8}$}  \text{there exists } C_*>0\text{ s.t. }\;\;\mathbf{F}^T D^2\mathcal{W}(\x,\mathbf I)\mathbf{F}\geq C_*|\mathbf{F}|^2\;\;\text{for all }\mathbf{F}\in\R^{3\times 3}\text{ and all }\x\in\Omega.
\end{equation}
Given a couple of external (dead) loads $\mathbf{f}\colon\Omega\to \R^3$ and $\mathbf{g}\colon\partial\Omega\to\R^3$ acting on the elastic body $\Omega$ as forces of body and surface type, respectively, equilibrium configurations are given by the deformations $\yy\colon\Omega\to\R^3$ which solve
\begin{equation}\label{eq:el_traction}
	\left\{
	\begin{aligned}
		-\mathrm{div}\,(D\W(\x,\nabla\yy))&=\mathbf{f}, &&\text{in }\Omega, \\
		D\W(\x,\nabla\yy)\mathbf{n}&=\mathbf{g},&&\text{on }\partial\Omega.
	\end{aligned}
	\right.
\end{equation}
 In view of the above assumptions on $\W$, it is trivial to observe that, if there are no external forces, i.e. $\mathbf f=\mathbf g=\mathbf 0$, then the identity map $\mathbf{i}$ solves \eqref{eq:el_traction}, \MMM as well as the map $\mathbf y(\x)=\mathbf R\x+\mathbf c$, for every $\mathbf R\in SO(3)$ and every $\mathbf c\in\R^3$.
\MMM Therefore one may expect that problem \eqref{eq:el_traction} admits a solution when the couple $(\mathbf{f},\mathbf{g})$ consists of a small perturbation of $(\mathbf{0},\mathbf{0})$ and that the gradient of such solution is close, in some suitable sense, to the identity matrix (or at least to a rotation). \KKK This approach is  based on a careful use of the implicit function theorem and it has been pursued for the first time in the works by Stoppelli \cite{Stoppelli54,Stoppelli55}, where  loads are of the form $(h\mathbf{f},h\mathbf{g})$, for some fixed $(\mathbf{f},\mathbf{g})$ and small parameter $h>0$. This approach for traction problems of the type \eqref{eq:el_traction} has been further developed by several authors, see e.g. \cite{CMW1,CMW2,LanzaValent,LeD}. We also mention the comprehensive book \cite{V} for a thorough description and \cite[Section 6.7]{C} for an exhaustive exposition of the literature concerning this topic (see also \cite{CGM,TN}).

We now outline the results, starting by fixing the functional framework.  We consider the Nemitsky operators associated with $-\mathrm{div}\,(D\W(\x,\nabla\yy(\x)))$ and $D\W(\x,\nabla\yy(\x))\mathbf{n}$, that is
\begin{equation}\label{eq:nemitsky}
	\begin{aligned}
		W^{2,p}(\Omega,\R^3)&\to L^p(\Omega,\R^3)\times W^{1-1/p,p}(\Omega,\R^3), \\
		\yy &\mapsto \left(\begin{gathered}
			-\mathrm{div}\,(D\W(\cdot,\nabla \yy(\cdot)))\\
			D\W(\x,\nabla\yy(\x))\mathbf{n}
		\end{gathered}\right).
	\end{aligned}
\end{equation}
We remark that it is hereafter fundamental to have $p>3$ in order to obtain \MMM that $W^{1,p}(\Omega,\R^3)$ is a Banach algebra along with \KKK the continuous embedding
\begin{equation}\label{eq:embedding}
	W^{1,p}(\Omega,\R^3)\hookrightarrow L^\infty(\Omega,\R^3).
\end{equation}
In particular, this ensures that the Nemitsky operator \eqref{eq:nemitsky} is of class $C^1$, see \cite[Lemma 2.1, Chapter V]{V}. This motivates   assumptions 
\eqref{intforces}
for external loads.
Given a small parameter $h>0$, we are then interested in existence and asymptotic behavior of solutions to the rescaled nonlinear problem
\eqref{eq:nonlinear_h}, \MMM which  formally coincide with critical points of the  functional $\mathcal G(\cdot;h\mathcal L)$  defined  by \eqref{functional}, with the load functional given by \eqref{lod}. \KKK
We easily notice that, if \eqref{eq:nonlinear_h}  possesses a solution, the forces must necessarily satisfy the following compatibility condition
\begin{equation}\label{eq:equil_1}
	\int_\Omega\mathbf{f}\,d\x+\int_{\partial\Omega}\mathbf{g}\,dS=0.
\end{equation}
In addition, after simple computations, one can observe that, if $\yy_h\in W^{2,p}(\Omega,\R^3)$ solves \eqref{eq:nonlinear_h}, then it must necessarily satisfy the zero moment condition
\begin{equation}\label{eq:equil_2}
	\int_\Omega\yy_h\wedge\mathbf{f}\,d\x+\int_{\partial\Omega}\yy_h\wedge\mathbf{g}\,dS=\mathbf 0.
\end{equation}
 When \eqref{eq:equil_2} holds, we say that the loads $\mathbf{f}$ and $\mathbf{g}$ are {equilibrated} with respect to the deformation $\yy_h$.  \eqref{eq:equil_1} can be imposed as an {a priori} condition on the loads while \eqref{eq:equil_2} may be only regarded as an {a posteriori} condition. Nevertheless, we  assume the loads $\mathbf{f}$ and $\mathbf{g}$ to satisfy
\begin{equation}\label{eq:equil_id}
	\int_\Omega\mathbf{x}\wedge\mathbf{f}\,d\x+\int_{\partial\Omega}\mathbf{x}\wedge\mathbf{g}\,dS=\mathbf0.
\end{equation}
{Loosely speaking, a justification of this assumption lies in the fact that a solution to \eqref{eq:nonlinear_h} (thus verifying \eqref{eq:equil_2}) is expected to be a small perturbation of the identity map $\mathbf{i}$. We just say that loads are equilibrated if \eqref{eq:equil_1} and \eqref{eq:equil_id} hold. Moreover, heuristically speaking, another natural condition which, together with \eqref{eq:equil_id}, goes in the direction of ruling out the chance that $\nabla\mathbf y_h$ converges to a rotation $\mathbf R\neq \mathbf I$ is the absence of axes of equilibrium for the couple $(\mathbf f,\mathbf g)$. We recall that $\mathbf{a}\in\R^3$ is said to be an \emph{axis of equilibrium} for $(\mathbf{f},\mathbf{g})$ if, for any rotation $\mathbf{R}\in SO(3)$ around $\mathbf{a}$, there holds
\[
\int_\Omega\mathbf{x}\wedge\mathbf{R}^T\mathbf{f}\,d\x+\int_{\partial\Omega}\mathbf{x}\wedge\mathbf{R}^T\mathbf{g}\,dS=0,
\]
which is equivalent to
\[ \int_\Omega\mathbf{R}\mathbf{x}\wedge\mathbf{f}\,d\x+\int_{\partial\Omega}\mathbf{R}\mathbf{x}\wedge\mathbf{g}\,dS=0.
\]
The presence of axes of equilibrium for a couple $(\mathbf{f},\mathbf g)$ is linked with the 
  \emph{astatic load matrix}, defined as
\begin{equation*}
	K_{\mathbf{f},\mathbf{g}}:=\int_\Omega \x\otimes \mathbf{f}\,d\x+\int_{\partial\Omega} \x\otimes \mathbf{g}\,dS.
\end{equation*}
In particular, if we call $\lambda_1,\lambda_2,\lambda_3\in\R$ the eigenvalues of the matrix $K_{\mathbf{f},\mathbf{g}}$ (which is symmetric in view of \eqref{eq:equil_id}), it is known that the absence of axes of equilibrium is equivalent to the following assumption:
\begin{equation}\label{eq:axis_equil}
	\lambda_i+\lambda_j\neq 0,\quad\text{for all }i,j=1,\dots,3,~i\neq j.
\end{equation}
We refer to \cite[§3]{CMW1} for a more detailed characterization of this property.}
\MMM
 We also observe that the couple of assumptions \eqref{eq:equil_1} and \eqref{eq:equil_id} coincides with \eqref{L1}.
Furthermore, we notice that if loads are equilibrated and satisfy \eqref{L1plus},  assumption \eqref{eq:axis_equil} (i.e., absence of axes of equilibrium) can be rephrased by prescribing that the {rotation kernel} \eqref{rotker} of $\mathcal{L}$
only contains the identity matrix  (this can be seen for instance by combining \cite[Proposition 6.2]{MM} and \cite[Proposition 3.8]{CMW1}).
\KKK

The following result is contained in \cite[Corollary 6.9, Chapter V]{V}. \MMM
We point out that for such a result the assumptions on the lack of axes of equilibrium for the loads $(\mathbf{f},\mathbf{g})$ could be relaxed, see for instance  Theorem 6.8 and Corollary 6.10 in Chapter V of \cite{V} and the thorough analysis in \cite{CMW1,CMW2}. 
However, to the best of our knowledge, even the results with weaker assumptions do not allow the full generality  for the equilibrated couple $(\mathbf{f},\mathbf{g})$: for instance, they do not cover the case in which every axis in $\mathbb R^3$ is an axis of equilibrium for $(\mathbf f,\mathbf g)$ (which is equivalent to saying that the rotation kernel of $\mathcal L$ is the whole of $SO(3)$), a case which we consider in our example in Section \ref{sectionexample}.
\KKK

\begin{theorem}[$\mbox{\cite[Corollary 6.9, Chapter V]{V}}$]\label{thm:inversion}
\MMM Suppose that $\Omega$ is of class $C^2$ and that \eqref{framind}, \eqref{Z1}, \eqref{w7} and \eqref{w8} are satisfied.  Let $p>3$. \KKK
	Let $\mathbf{f}\in L^p(\Omega,\R^3)$ and $\mathbf{g}\in W^{1-1/p,p}(\partial\Omega,\R^3)$ satisfy \eqref{eq:equil_1} and \eqref{eq:equil_id} and let us assume that the eigenvalues $(\lambda_i)_{i=1,\dots,3}$ of the astatic load  $K_{\mathbf{f},\mathbf{g}}$ satisfy \eqref{eq:axis_equil}. Then there exists $h_0>0$ such that, for any $h\in (0,h_0)$ there exists a deformation $\yy_h\in W^{2,p}(\Omega,\R^3)$ which solves \eqref{eq:nonlinear_h}. Moreover the map
	\begin{equation}\label{eq:pert_map}
		\begin{aligned}
			\mathbf{Y}\colon  [0,h_0) &\to W^{2,p}(\Omega,\R^3), \\
			h &\mapsto \mathbf{Y}(h)=\yy_h
		\end{aligned}
	\end{equation}
	extended to the identity $\mathbf{i}$ for $h=0$, belongs to $C^1([0,h_0);W^{2,p}(\Omega,\R^3))$.
\end{theorem}

It is now immediate to investigate the asymptotic behavior of the deformation $\yy_h$ as $h\to 0$. A direct consequence of Theorem \ref{thm:inversion} is that
\[
\yy_h\to \mathbf{i}\quad\text{in }W^{2,p}(\Omega,\R^3)~\text{as }h\to 0
\]
which, in view of \eqref{eq:embedding}, in turn implies that
\[
\nabla \yy_h\to \mathbf{I},\quad\in L^\infty(\Omega,\R^3)~\text{as }h\to 0.
\]
Being the map $\mathbf{Y}$ from \eqref{eq:pert_map} of class $C^1$, we can perform a first order Taylor expansion at $h=0$, which guarantees the existence of a function $\mathbf{u}_0\in W^{2,p}(\Omega,\R^3)$ such that
\[
\u_h:=\frac{\yy_h-\mathbf{i}}{h}\to 	\mathbf{u}_0\quad\text{in }W^{2,p}(\Omega,\R^3)~\text{as }h\to 0.
\]
Since
\[
h\nabla \u_h\to 0\quad\text{in }L^\infty(\Omega,\R^3)~\text{as }h\to 0,
\]
we can carry out a Taylor expansion of \MMM $D\mathcal W(\x,\cdot)$ near the identity, that is
\begin{equation}\label{eq:taylor}
		D\W(\x,\mathbf{I}+h\nabla\u_h(\x))= h\,D^2\W(\x,\bm{\xi}_h(\x))\nabla \u_h(\x)\quad\text{for every }\x\in\Omega,
\end{equation}
for some $\bm{\xi}_h(\x)$ belonging to the segment $[\mathbf{I},\mathbf{I}+h\nabla \u_h(\x)]$, and thus $\bm{\xi}_h$ is \KKK uniformly converging to $\mathbf{I}$ as $h\to 0^+$. Now, in view of the equation \eqref{eq:nonlinear_h} satisfied by $\yy_h=\mathbf i+h\u_h$ and of \eqref{eq:taylor} we deduce that
\begin{equation*}
	\int_\Omega D^2\W(\x,\bm{\xi}_h(\x))\nabla \u_h(\x):\nabla \v\,d\x=\int_\Omega \mathbf{f}\cdot\v\,d\x+\int_{\partial\Omega}\mathbf g\cdot\v\,dS
\end{equation*}
for all $\v\in H^1(\Omega,\R^3)$. Then, thanks to the regularity of $\mathcal{W}$ we can pass to the limit as $h\to 0$ and obtain
\begin{equation*}
	\int_\Omega D^2\W(\x,\mathbf{I})\nabla \u_0(\x):\nabla \v\,d\x=\int_\Omega \mathbf{f}\cdot\v\,d\x+\int_{\partial\Omega}\mathbf g\cdot\v\,dS
\end{equation*}
for all $\v\in H^1(\Omega,\R^3)$, i.e. $\u_0$ weakly solves
\begin{equation*}
	\left\{\begin{aligned}
		-\mathrm{div}\,(D^2\W(\x,\mathbf{I})\nabla \u_0)&=\mathbf{f}&&\text{in }\Omega, \\
		(D^2\W(\x,\mathbf{I})\nabla\u_0)\mathbf{n}&=\mathbf{g}&&\text{on }\partial\Omega.
	\end{aligned}\right.
\end{equation*}
This means that $\u_0$ is an equilibrium configuration for the linearized elastic problem with loads $(\mathbf{f},\mathbf{g})$ and as such it minimizes functional $\mathcal F_0(\cdot;\mathcal L)$ over $H^1(\om,\R^3)$. \KKK

\section{Proof of the main results}\label{proofs}

Let us state a preliminary result about Korn inequalities. \MMM We state it  for  $W^{1,p}(\om,\mathbb R^3)$ vector fields with generic $p\in(1,+\infty)$. \KKK It is based on the standard  Korn inequality, see \cite{N}, which yields the existence of a constant $D_{p,\om}$ such that for every $\u\in W^{1,p}(\om,\mathbb R^{3\times3})$ there holds
\begin{equation}\label{2ndkorn}
\min_{\mathbf W\in\mathbb R^{3\times3}_{\mathrm{skew}}}\|\nabla\u-\mathbf W\|_{L^p(\om,\mathbb R^{3\times3})}\le D_{p,\om}\,\|\mathbb E\u\|_{L^p(\om,\mathbb R^{3\times3})}.
\end{equation}

\begin{lemma}\label{kornrot} Let $p\in(1,+\infty)$.
There are constants $Z_{p,\om}$ and $Q_{p,\om}$ (only depending on $p,\om$) such that
\begin{equation}\label{korn1}
\int_\om|\nabla \u|^p\le Z_{p,\om}\int_\om |\mathbb E\u|^p\quad\mbox{for every $\u\in W^{1,p}(\om,\mathbb R^3)$ s.t. $\displaystyle\int_\om\mathrm{curl}\,\u=0$} 
\end{equation}
and
\begin{equation}\label{korn2}
\int_\om|\nabla \u|^p\le Q_{p,\om}\int_\om |\mathbb E\u|^p\quad\mbox{for every $\u\in W^{1,p}(\om,\mathbb R^3)$ s.t. $\displaystyle\int_\om|\nabla\u|^{p-2}\mathrm{curl}\,\u=0$}. 
\end{equation}
\end{lemma}
\begin{proof}
Suppose first that $\displaystyle\int_\om|\nabla\u|^{p-2}\mathrm{curl}\,\u=0$. Then for every $\mathbf W\in\R^{3\times3}_{\mathrm{skew}}$
\[
\frac{d}{d\eps}{\bigg{|}}_{\eps=0}\int_\om|\nabla\u+\eps\mathbf W|^p=p\int_\om|\nabla\u|^{p-2} \nabla\u:\mathbf W=p\int_\om|\nabla\u|^{p-2} \,\mathrm{skew}(\nabla\u):\mathbf W=0,
\]
implying that the unique projection of $\nabla\u$ on the vector space of constant skew-symmetric tensor fields over $\Omega$ is $\mathbf 0$. Therefore, \eqref{2ndkorn} gives \eqref{korn2}.

In order to prove that \eqref{korn1} holds, assume by contradiction that there exists a sequence $(\u_j)_{j\in\mathbb N}\subset W^{1,p}(\om,\R^3)$ such that $\int_\om\mathrm{curl}\,\u_j=\mathbf 0$ and $\int_\om|\nabla\u_j|^p=1$ for every $j\in\mathbb N$, and $\int_\om|\mathbb E\u_j|^p\to0$ as $j\to\infty$.
By \eqref{2ndkorn} and since $\int_\om \mathrm{skew}(\nabla\u_j)=\mathbf 0$, there exists a sequence $(\mathbf W_j)_{j\in\mathbb N}\subset\mathbb R^{3\times3}_{\mathrm{skew}}$ such that
\[\begin{aligned}
|\om||\mathbf W_j|&=\left|\int_\om (\mathrm{skew}(\nabla\u_j)-\mathbf W_j)\right|\le |\om|^{\frac{p-1}{p}}\|\nabla\u_j-\mathbf W_j\|_{L^p(\om,\mathbb R^{3\times3})}\\&\le |\om|^{\frac{p-1}{p}} D_{p,\om}\,\|\mathbf E\u_j\|_{L^p(\om,\mathbb R^{3\times3})},
\end{aligned}\] 
and since $\int_\om|\mathbb E\u_j|^p\to0$ we conclude that $\mathbf W_j\to\mathbf 0$ and that $\int_\om|\nabla\u_j-\mathbf W_j|^p\to 0$ as $j\to\infty$. But then we have
\[
\int_\om|\nabla\u_j|^p\le 2^{p-1}\int_\om|\nabla\u_j-\mathbf W_j|^p+2^{p-1}|\om||\mathbf W_j|^p
\]
and the right hand side goes to $0$ as $j\to\infty$,  contradicting the fact that $\int_\om|\nabla\u_j|^p=1$ for every $j\in\mathbb N$ and thus proving \eqref{korn1}.
\end{proof}

 Let us now state a simple result about the minimization of the linear elastic energy.\KKK

 \begin{lemma}\label{basiclemma} \MMM Let $M>0$. \KKK
  Suppose that \eqref{framind}, \eqref{Z1}, \eqref{reg}, \eqref{coerc}, hold and  that $\mathcal L$ satisfies \eqref{L1}
 and
$\|\mathcal L\|_* \le {C\MMM M \KKK}/{K_\Omega}$,
where  $K_\Omega$ is the constant in \eqref{KP} and $C$ is the constant in \eqref{coerc}.
  If $\u_*\in H^1(\om,\mathbb R^3)$ minimizes $\mathcal F_0(\cdot;\mathcal L)$ over $H^1(\om,\R^3)$, then  $\int_\om |\mathbb E\u_*|^2\,d\x\le \MMM M^2\KKK$.
  \end{lemma}
  \begin{proof}
  Let us consider a solution $\u_*\in H^1(\om,\mathbb R^3)$ to the linear elastic problem
\begin{equation}\label{lineare}\min\left\{\mathcal F_0(\cdot;\mathcal L): \u\in H^1(\om,\mathbb R^3)\right\},\end{equation}
whose existence is ensured by the compatibility conditions \eqref{L1}  by standard arguments. Being a minimizer, $\u_*$  satisfies the energy identity
\begin{equation}\label{energyidentity}
\int_\om\mathbb E\u_*\,D^2\mathcal W(\x,\mathbf I)\,\mathbb E\u_*=\mathcal L(\u_*).
\end{equation} 
From  \eqref{energyidentity}, by the basic ellipticity estimate $\mathrm{sym}\mathbf F\,D^2\mathcal W(\x,\mathbf I)\,\mathrm{sym}\mathbf F\ge C|\mathrm{sym}\mathbf F|^2,$ where $C$ is the constant in \eqref{coerc}, by taking into account \eqref{KP}  we get
\[\begin{aligned}
 C\int_\om|\mathbb E\u_*|^2\le \int_\om\mathbb E\u_*\,D^2\mathcal W(\x,\mathbf I)\,\mathbb E\u_*=\mathcal L(\u_*)\le K_\Omega\|\mathcal L\|_{*}\,\|\mathbb E\u_*\|_{L^{2}(\om,\mathbb R^{3\times3})},
\end{aligned}\]
that is,
$$
 \|\mathbb E\u_*\|_{L^{2}(\om,\mathbb R^{3\times3})}^2\le C^{-2}K_\om^2\,\|\mathcal L\|_{*}^2.
$$
Thanks to the  the assumption $\|\mathcal L\|_* \le {C\MMM M\KKK}/{K_\Omega}$, we conclude.
  \end{proof}

For every $R > 0$ and for every  $\mathcal L\in (H^1(\om,\mathbb R^3))^*$, in the sequel we shall  consider the following constrained minimization problems of finite elasticity
\begin{equation}\label{min}
\min\left\{\mathcal F(\u; \mathcal L):\u\in H^1(\om,\R^3),\,\int_\om |\mathbb E\u|^2\le R\right\}
\end{equation}
\begin{equation}\label{minrot}
\min\left\{\mathcal F(\u; \mathcal L):\u\in H^1(\om,\R^3),\,\int_\om |\mathbb E\u|^2\le R,\, \int_\om\mathrm{curl}\, \u=\mathbf0\right\}
\end{equation}


\begin{lemma}\label{exist}
Suppose that \eqref{framind}, \eqref{Z1}, \eqref{reg}, \eqref{coerc}, \eqref{coerc2}, \eqref{poly} hold 
and that  $\mathcal L \in (H^1(\om,\mathbb R^3))^*$ satisfies \eqref{L1}. Let $R>0$. Then  there exists a solution to problem \eqref{min} and there exists a solution to problem \eqref{minrot}. 
\end{lemma}

\begin{proof} 

Let $(\u_n)_{n\in\mathbb N}\subset H^1(\om,\mathbb R^3)$ denote a minimizing sequence for problem \eqref{min}. Since $\mathcal F(\mathbf 0;\mathcal L)=0$, it is not restrictive to assume that $\mathcal F(\u_n;\mathcal L)\le 1$ for every $n\in\mathbb N$. 
Moreover, by \ref{rigidity}  there exists a sequence $(\mathbf R_n)_{n\in\mathbb N}\subset SO(3)$ such that
\begin{equation}\label{zero}\begin{aligned}
\int_\om |\mathbf I+\nabla\u_n-\mathbf R_n|^2\,d\x 
&\le  \,C_{\om}\int_\om \mathcal W(\x,\mathbf I+\nabla\u_n)\,d\x\\&\le C_{\om} \mathcal F(\u_n)+C_{\om} \mathcal L( \u_n)\le C_{\om}+C_{\om} \mathcal L( \u_n).
\end{aligned}\end{equation}
 From \eqref{KP} and \eqref{zero}, since $\int_\om|\mathbb E\u_n|^2\le R$,
we deduce that the sequence $(\nabla \u_n)_{n\in\mathbb N}$ is  bounded in $L^2(\om,\mathbb R^{3\times3})$. Thanks to Poincar\'e inequality and to \eqref{L1}, we deduce that up to subsequences  $\nabla \u_n$ weakly converge to $\nabla\u$ for some $\u\in H^1(\om,\mathbb R^3)$,  that $\mathcal L( \u_n)\to\mathcal L( \u)$, and  that $\mathbb E\u_n$ weakly converge to $\mathbb E\u$ in $L^2(\om,\mathbb R^3)$. The  weak $L^2(\om,\mathbb R^{3\times3})$ lower semicontinuity of the map $\mathbf F\mapsto\int_\om |\mathbf F|^2$ yields  $\int_\om |\mathbb E\u|^2\le R$.  
 On the other hand, with the notation $\mathbf y_n(\x):=\x+h\u_n(\x)$, the assumption \eqref{coerc2} implies, thanks to the {classical results by  Ball \cite[§6]{B}}, that $\mathrm{cof}\nabla\mathbf y_n$ (resp. $\det\nabla\mathbf y_n$) weakly converge in $L^{q}(\om,\mathbb R^{3\times3})$ to $\mathrm{cof}\nabla\mathbf y$ (resp. weakly converge in $L^{r}(\om)$ to $\det\nabla\mathbf y$).
By the polyconvexity assumption \ref{poly} on the Carath\'eodory function  $\mathcal W$, and since $\mathcal L( \u_n)\to\mathcal L( \u)$ as $n\to\infty$, we deduce that
\[
\mathcal F(\u;\mathcal L)\le\liminf_{n\to+\infty}\mathcal F(\u_n;\mathcal L).
\]
Since $(\u_n)_n$ is by assumption a minimizing sequence, we conclude that $\u\in H^1(\om,\mathbb R^3)$ is a solution to problem \eqref{min}.

The same argument shows that there exists a solution to problem \eqref{minrot}, since the weak $L^2(\om,\mathbb R^{3\times3})$ convergence of $\nabla\u_n$ to $ \nabla\u$  implies that $\int_\om\mathrm{curl}\,\u_n\to\int_\om\mathrm{curl}\,\u$. 
\end{proof}
\begin{remark}{\rm It is worth noticing that a solution $\u$ to problem \eqref{minrot} is not a priori guaranteed to be a constrained local minimizer in the sense of Definition \ref{constrained}. Indeed it may happen that  $\u$ satisfies
$\int_\om |\mathbb E\u|^2= R,$ so that
for instance by taking  $\boldsymbol\psi=\u$, for every $\eps>0$ the function $\u+\eps\boldsymbol\psi$ is not admissible for problem \eqref{minrot}. 
 }
\end{remark}
\begin{remark}\label{minremark} {\rm Let $h>0$ \MMM and $M>0$\KKK. Then by Lemma \ref{exist} there exists a solution $\v_h$ to the problem
 \begin{equation*}
\min \left \{\mathcal F( \v; h\mathcal L): \v\in H^1(\om,\mathbb R^3),\ \int_\om |\mathbb E\v|^2\le \MMM M^2\KKK h^2,\,\int_\om\mathrm{curl}\,\v=0\right \},
\end{equation*}
 and thus $\u_h=:h^{-1}\v_h$ is a solution to
  \beeq 
  \lab{uh} 
  \min \left \{\mathcal F_{h}( \u; \mathcal L): \u\in H^1(\om,\mathbb R^3),\ \int_\om |\mathbb E\u|^2\le\MMM M^2\KKK,\,\int_\om\mathrm{curl}\,\u=\mathbf 0\right \}.\eneq}
  \end{remark}

The next lemma provides convergence to \MMM constrained \KKK linear elasticity and generalizes results from \cite{MPTARMA,MP2}.  
  
\begin{lemma}
\label{conv} 
Suppose that \eqref{framind}, \eqref{Z1}, \eqref{reg}, \eqref{coerc}, \eqref{coerc2}, \eqref{poly} hold and  that $\mathcal L\in (H^1(\om,\R^3))^*$ satisfies \eqref{L1}.
\MMM Let $M>0$ and  \KKK let   $(h_{j})_{j\in\mathbb N}\subset(0,1)$ be a vanishing sequence.  For any $j\in\mathbb N$, let  $\u_j:=\u_{h_j}$, where $\u_{h_j}$ is a solution to problem \eqref{uh} (with $h_j$ in place of $h$, see {\rm Remark \ref{minremark}}). 
Then there exists 
$\u_ *\in H^1(\Omega,\R^3)$ 
 such that $\int_\om |\mathbb E\u_*|^2\le\MMM M^2\KKK$ and such that, up to subsequences,
$
\mathbb E\u_j\wconv\mathbb E\u_*$
 weakly in  $L^2(\Omega,\mathbb R^{3\times 3})$  as $j\to\infty$ and 
\beeq\label{limite1}
\lim_{j\to +\infty}\mathcal F_{h_{j}}(\u_j;\mathcal L)= \mathcal F_0(\u_*\,; \mathcal L)
=\min_{\u\in H^1(\om,\mathbb R^3)}\left\{\mathcal F_0(\u\,; \mathcal L): \int_\om |\mathbb E\u|^2\le \MMM M^2\KKK\right\}.
\end{equation}
\end{lemma}
\begin{proof} \textbf{Step 1 (lower bound).} 
Suppose  that $\sup_{j\in\mathbb N}\F_{h_j}(\tilde\u_j;\mathcal L)<+\infty$  and that $
\mathbb E\tilde\u_j\wconv\mathbb E\tilde\u$ weakly in  $L^2(\Omega,\mathbb R^{3\times 3})$ for some sequence $(\tilde\u_j)_{j\in\mathbb N}\subset H^1(\om,\mathbb R^3)$  such that $\int_\om\mathrm{curl}\,\tilde\u_j=\mathbf 0$ for every $j\in\mathbb N$ and some $\tilde \u\in H^1(\om,\mathbb R^3)$.  By the  weak $L^2(\om,\mathbb R^{3\times3})$ lower semicontinuity of the map $\mathbf F\mapsto\int_\om |\mathbf F|^2$, we get
\begin{equation}\label{convexity}\liminf_{j\to +\infty}\int_\om |\mathbb E\tilde\u_j|^2\ge  \int_\om |\mathbb E\tilde\u|^2.\end{equation}
%

By applying \cite[Lemma 5.1]{MP2} we obtain the existence of $\mathbf W\in \R^{3\times3}_{\mathrm{skew}}$ such that $\sqrt{h_j}\nabla\tilde\u_j\to\mathbf W$ in $L^2(\om,\mathbb R^{3\times3})$ as $j\to \infty$.  However, since $\int_\om\mathrm{curl}\,\tilde\u_j=\mathbf 0$, by Lemma \ref{kornrot} we deduce that the sequence $(\nabla\tilde\u_j)_{j\in\mathbb N}$ is bounded in $L^2(\om,\R^3)$, thus forcing $\mathbf W\equiv\mathbf 0$. From \cite[Lemma 5.2]{MP2} we also deduce that  $\mathbf 1_{B_j}\mathbb E(\tilde\u_j)\rightharpoonup\mathbb E(\tilde \u)$ weakly in $L^2(\om,\R^{3\times3})$ as $j\to\infty$, where $B_j:=\{\x\in\om:|\sqrt{h_j}\nabla\tilde\u_j(\x)|\le 1\}$. We stress that \cite[Lemma 5.1, Lemma 5.2]{MP2} are stated under the additional incompressibility constraint requiring $\mathcal W(\mathbf F)=+\infty$ if $\det\mathbf F\neq1$, but they are true (with the very same proof) even without such a constraint.

By the frame 
 indifference assumption \eqref{framind}, there exists a function $\mathcal V$ such that for a.e. $\x\in\om$
\begin{equation*}
\W(\x,\mathbf F)=\V(\x,\textstyle{\frac{1}{2}}( \mathbf F^T \mathbf F - \mathbf I))\,
\qquad
\ \forall\, \mathbf F\in \mathbb R^{3\times 3}.
\end{equation*}
By \eqref{reg} we deduce that for a.e. $\x\in\om$ the function $\mathcal V(\x,\cdot)$ is $C^2$ smooth in a suitable neighborhood $\tilde{\mathcal U}$ of the origin in $\R^{3\times3}$, with a $\x$-independent modulus of continuity $\eta:\mathbb R_+\to\mathbb R$, which is increasing and vanishing at $0^+$.
By following the proof of \cite[Lemma 5.3]{MP2}, we define $$\mathbb D\tilde \u_j:=\mathbb E\tilde \u_j+\frac12h_j\nabla\tilde\u_j^T\nabla\tilde\u_j.$$
 By \eqref{Z1}, \eqref{reg} and by a Taylor expansion, we get the existence of $j_0\in \mathbb N$ such that for every $j>j_0$ the estimate
\begin{equation*}
\left|\mathcal V(\x, h_j\mathbb D\tilde \u_j)- \frac{h_j^2}{2} \, \mathbb D\tilde\u_j^T\, D^2\mathcal V (\x, \mathbf 0) \ \mathbb D\tilde\u_j \right|\le h_j^2\eta(h_j|\mathbb D\tilde\u_j|)|\mathbb D\tilde\u_j|^2
\end{equation*}
holds for a.e.  $\x\in B_j$, since $h_j|\mathbb D\tilde\u_j(\x)|\le2\sqrt{h_j}$ and thus $h_j \mathbb D\tilde\u_j(\x)\in \tilde{\mathcal U}$ for a.e. $\x\in B_j$ if $j$ is large enough. We deduce
\begin{equation*}\begin{aligned}
\mathcal F_{h_j}(\tilde\u_j)&\ge \frac1{h_j^2}\int_{B_j}\mathcal W(\x, \mathbf I+h_j\nabla\tilde\u_j)\,d\x-\mathcal L(\tilde\u_j)=\frac1{h_j^2}\int_{B_j}\mathcal V(\x,h_j \mathbb D\tilde\u_j)\,d\x-\mathcal L(\tilde\u_j)\\
&\displaystyle\ge \int_{B_j}\frac12 \mathbb D\tilde\u_j^T\,D^2\mathcal V(\x,\mathbf 0)\,\mathbb D\tilde\u_j\,d\x-\int_{B_j} \eta(h_j|\mathbb D\tilde\u_j|)|\mathbb D\tilde\u_j|^2\,d\x-\mathcal L(\tilde\u_j)\\
&\ge\frac12\int_\om (\mathbf 1_{B_j}\mathbb D\tilde\u_j)\,D^2\mathcal W(\x,\mathbf I)\,(\mathbf 1_{B_j}\mathbb D\tilde\u_j)\,d\x-\eta(2\sqrt{h_j})\int_\om|\mathbf 1_{B_j}\mathbb D\tilde\u_j|^2\,d\x-\mathcal L(\tilde\u_j)
\end{aligned}
\end{equation*}
for every $j>j_0$,
where the last inequality is due $D^2\mathcal W(\x,\mathbf I)=D^2\mathcal V(\x,\mathbf0)$ and to the monotonicity of the modulus of continuity $\eta$, and we also recall that $\eta(2\sqrt{h_j})\to 0$ as $j\to \infty$. 
By recalling that $\sqrt{h_j}\nabla\tilde\u_j$ strongly converge to zero in $L^2(\om,\R^{3\times3})$, hence a.e. in $\om$ up to subsequences, and since $(\mathbf 1_{B_j}h_j\nabla\tilde\u_j^T\nabla\tilde\u_j)_{j\in\mathbb N}$ is a bounded sequence in $L^2(\om,\mathbb R^{3\times3})$, we deduce that up to subsequences $\mathbf 1_{B_j}h_j\nabla\tilde\u_j^T\nabla\tilde\u_j$ weakly converge to zero in $L^2(\om,R^{3\times3})$ so that $\mathbf 1_{B_j}\mathbb D\tilde\u_j$ weakly converge in $L^2(\om,\mathbb R^{3\times3})$ to $\mathbb E\tilde\u$ as $j\to\infty$.
Since the map $\mathbf F\mapsto\int_\om\mathbf F^T\, D^2\mathcal W(\x,\mathbf I)\,\mathbf F\,d\x$ is weakly $L^2(\om,\R^{3\times3})$ lower semicontinuous, and since $\mathcal L\in (H^1(\om,\R^3))^*$  and \eqref{L1} imply  $\mathcal L(\tilde\u_j )\to\mathcal L(\tilde\u)$ thanks to Korn and Poincar\'e inequalities, we conclude that
 \begin{equation*}
\liminf_{j\to +\infty} \F_{h_j}(\tilde\u_j;\mathcal L)\ge\frac12\int_\Omega \mathbb E\tilde\u\,D^2\mathcal W(\x,\mathbf I)\,\mathbb E\tilde\u\, d\x- \mathcal L(\tilde \u).
\end{equation*}

{\textbf {Step 2 (upper bound).}}
Let now $\bar\u\in H^1(\om,\mathbb R^3)$ be such that $\int_{\om} |\mathbb E\bar\u|^2\le \MMM M^2\KKK$. 
Let  $\delta_j:=h_j^{1/5}$ and $\bar\u_j:=\bar\u\ast\rho_j$, where $\rho_j(\x):=\delta_j^{-3}\rho(\x/\delta_j)$ and $\rho:\mathbb R^3\to \mathbb R$ is the standard unit symmetric mollifier.
We notice that $\bar\u_j\to\bar\u$ in $H^1(\om,\mathbb R^3)$ as $j\to+\infty $ and that the elementary estimate $\|\rho_j\|_{W^{1,\infty}(\mathbb R^3)}\le 2\delta_j^{-4}\|\rho\|_{W^{1,\infty}(\mathbb R^3)}$ holds. Therefore, by Young inequality we obtain
\[
\|\bar \u_j\|_{W^{1,\infty}(\om,\mathbb R^3)}\le \|\bar\u\|_{L^1(\om',\mathbb R^3)}\|\rho_j\|_{W^{1,\infty}(\mathbb R^3)}\le2\delta_j^{-4}\|\rho\|_{W^{1,\infty}(\mathbb R^3)}\|\bar\u\|_{L^1(\om',\mathbb R^3)}
\] 
where $\om'$ is a suitable open neighbor of $\om$ (and $\bar\u$ is understood as a not relabeled $H^1(\om',\mathbb R^3)$ extension).
We deduce 
$h\|\bar\u_{j}\|_{W^{1,\infty}(\om,\mathbb R^3)}\to 0$ as $j\to +\infty$.
 Therefore  $\mathbf I+h_j\nabla\bar\u_j\in\mathcal U$ for a.e. $\x$ in $\Omega$ if $j$ is large enough, where $\mathcal U$ is the neighborhood of $SO(3)$ that appears in assumption \eqref{reg}, which then implies, by  Taylor expansion, 
\begin{equation*}\label{nee}\begin{aligned}
\limsup_{j\to+\infty} \mathcal |\mathcal F_{h_j}(\bar\u_j;\mathcal L)- \mathcal F_0(\bar\u_j;\mathcal L)|&\le\limsup_{j\to+\infty}\int_{\Omega} \left|\frac{1}{h_j^2}\mathcal W(\x,\mathbf I+h_j\nabla\bar\u_j)-\frac12\,\nabla\bar\u_j^T D^2\mathcal W(\x,\mathbf I) \nabla\bar\u_j\right|\\
&\le\limsup_{j\to+\infty}\int_{\Omega}\omega(h_j|\nabla\bar\u_j|)\,|\nabla\bar\u_j|^2\\&\le\limsup_{j\to+\infty}\,\|\omega(h_j|\nabla \bar\u_j|)\|_{L^\infty(\om)}\int_{\om} |\nabla\bar\u_j|^2 =0.
\end{aligned}\end{equation*}
The latter limit is zero since $h_j\nabla \bar\u_j\to0$ in $L^\infty(\Omega,\mathbb R^{3\times3})$, as $\omega$ is increasing with $\lim_{t\to0^+}\omega(t)= 0$,  and since $\bar\u_j\to\bar \u$ in $H^1(\Omega,\mathbb R^3)$ as $j\to+\infty$, which also implies
 $\mathcal F_0(\bar\u_j;\mathcal L)\to \mathcal F_0(\bar\u;\mathcal L)$, so that by writing 
 $|\mathcal F_{h_j}(\bar\u_j;\mathcal L)- \mathcal F_0(\bar\u;\mathcal L)|\le  |\mathcal F_{h_j}(\bar\u_j;\mathcal L)- \mathcal F_0(\bar\u_j;\mathcal L)|+
|\mathcal F_0(\bar\u_j;\mathcal L)- \mathcal F_0(\bar\u;\mathcal L)|$ we deduce
%
\[\displaystyle\lim_{j\to +\infty} \F_{h_j}(\bar\u_j;\mathcal L)= \F_0(\bar\u;\mathcal L).\]
Moreover, by using Young inequality again we deduce that for any $j\in\mathbb N$
 \begin{equation}\label{convu*}
 \int_\om |\mathbb E\bar\u_j|^2\le \int_\om|\mathbb E\bar\u|^2\le M^2.
 \end{equation}

{\textbf {Step 3 (convergence).}}
Since we are assuming 
 \begin{equation}\label{ujmin}\u_j\in\argmin_{\u\in H^1(\om,\mathbb R^3)} \left \{\mathcal F_{h_{j}}( \u; \mathcal L):  \int_\om |\mathbb E\u|^2\le \MMM M^2\KKK ,\,\int_\om\mathrm{curl}\,\u=\mathbf0\right \}\end{equation}
for every $j\in\mathbb N$,
 it is readily seen that  that $\sup_{j\in\mathbb N}\F_{h_j}(\u_j;\mathcal L)<+\infty$. Moreover, since $\int_\om |\mathbb E\u_j|^2\le M^2$ for every $j\in\mathbb N$,
up to subsequences there exists $\u_*\in H^1(\om,\mathbb R^3)$ such that $
\mathbb E(\u_j)\wconv\mathbb E(\u_*)$ weakly in  $L^2(\Omega,\mathbb R^{3\times 3})$, see  \cite[Lemma 3.2]{MPTARMA}. By  Step 1 we get $\int_\om |\mathbb E\u_*|^2\le \MMM M^2\KKK$, see \eqref{convexity}, and
\begin{equation}\label{LIMINF}
\frac12\int_\Omega \mathbb E(\u_*)\,D^2\mathcal W(\x,\mathbf I)\,\mathbb E\u_*\, d\x- \mathcal L( \u_*)\le \liminf_{j\to +\infty} \F_{h_j}(\u_j;\mathcal L).
\end{equation}
 To every $\bar\u\in H^{1}(\om,\mathbb R^3)$ such that $\int_\om |\mathbb E\bar\u|^2\le \MMM M^2\KKK$,
 we associate the sequence $(\bar\u_j)_{j\in\mathbb N}\subset H^{1}(\om,\mathbb R^3)$ constructed as in Step 2. Therefore, such a sequence  converges to $\bar \u$ strongly in $H^1(\om,\mathbb R^3)$ with $h_j\|\nabla\bar\u_j\|_{L^{\infty}(\om,\R^{3\times3})}\to 0$ as $j\to+\infty$, and it satisfies $\int_\om |\mathbb E\bar\u_j|^2\le \MMM M^2\KKK $ for every  $j\in\mathbb N$, see \eqref{convu*}.
 For every $j\in\mathbb N$ we let   $\mathbf w_j:=|\om|^{-1}\int_\om\mathrm{curl}\,\bar\u_j$ so that $$\int_\om\mathrm{curl}\,(\bar\u_j-\tfrac12\w_j\wedge\x)=0,$$ since $\mathrm{curl}(\mathbf a\wedge\x)=2\mathbf a$ for every $\mathbf a\in \R^3$. Thus $\mathbf w_j\to\mathbf w:=|\om|^{-1}\int_\om\mathrm{curl}\,\bar\u$ as $j\to\infty$.
  We deduce that $\bar\u_j-\tfrac12\mathbf w_j\wedge\x$ strongly converge in $H^1(\om,\R^3)$ to $\bar\u-\tfrac12\mathbf w\wedge\x$ as $j\to+\infty$, that $h_j\|\nabla(\bar\u_j-\tfrac12\mathbf w_j\wedge\x)\|_{L^{\infty}(\om,\R^{3\times3})}\to 0$ as $j\to\infty$, and that for every $j\in\mathbb N$ there holds $$\int_\om |\mathbb E(\bar\u_j-\tfrac12\mathbf w_j\wedge\x)|^2=\int_\om |\mathbb E\bar\u_j|^2\le \MMM M^2\KKK.$$

   By \eqref{LIMINF}, \eqref{ujmin} and by Step 2 we get
\begin{equation*}\begin{aligned}
\mathcal F_0(\u_*;\mathcal L)&\le \liminf_{j\to +\infty} \F_{h_j}(\u_j;\mathcal L)\le\limsup_{j\to +\infty} \F_{h_j}(\u_{j};\mathcal L)\\&\le \limsup_{j\to +\infty} \F_{h_j}(\bar\u_{j}-\tfrac12\mathbf w_j\wedge\x;\mathcal L)=\mathcal F_{0}(\bar\u-\tfrac12\mathbf w\wedge \x;\mathcal L)=\mathcal F_0(\bar\u;\mathcal L),
\end{aligned}\end{equation*}
where the last equality is due to the invariance of $\mathcal F_0(\cdot;\mathcal L)$ by infinitesimal rigid displacements.
By the arbitrariness of $\bar\u$, we deduce that $\u_*$ is a solution to the minimization problem in the right hand side of \eqref{limite1}, and that indeed \eqref{limite1} holds.
\end{proof}
\begin{lemma}\lab{intmin}
Assume that \eqref{framind}, \eqref{Z1}, \eqref{reg}, \eqref{coerc}, \eqref{coerc2}, \eqref{poly} hold true,
 that  $\mathcal L \in (H^1(\om,\mathbb R^3))^*$ satisfies \eqref{L1} and that \MMM
\begin{equation}\label{h2}
\|\mathcal L\|_* <\gamma (\om,C):= \frac{1}{K_{\om}}\min\left\{\frac1{C_\om}, \frac C2\right\}
\end{equation}
where $C_{\om}$ is the constant in \eqref{rigidity}, $K_\Omega$ is the constant in \eqref{KP} and $C$ is the constant in \eqref{coerc}.
Then there exists \MMM$\bar h>0$  \KKK such that if  $h\in(0,\MMM\bar h\KKK)$  and if 
 \beeq\lab{minh}\v_h\in\argmin \left \{\mathcal F( \v; h\mathcal L): \v\in H^1(\om,\mathbb R^3),\ \int_\om |\mathbb E(\v)|^2\le h^2,\,\int_\om\mathrm{curl}\,\v=\mathbf0\right \},
 \eneq
 then $\displaystyle\int_\om |\mathbb E \v_h|^2<h^2$. 
\end{lemma}
\begin{proof}
{\textbf {Step 1.}} Let us consider a minimizer $\u_*\in H^1(\om,\mathbb R^3)$ of the linear elastic problem \eqref{lineare},
whose existence is ensured by the compatibility conditions \eqref{L1} on $\mathcal L$. It  satisfies the energy identity
\eqref{energyidentity},
and since by \eqref{h2} we have $\|\mathcal L\|_*\le \tfrac{C}{2K_\Omega}$, by applying  Lemma \ref{basiclemma} with $M=1/2$ we deduce
 $\int_{\om}|\mathbb E\u_*|^2\le 1/4$. 
Let $h\in(0,1)$ and let $\bar\u_h:=\u_*\ast\rho_h$, where $\rho_h(\x)=\delta_h^{-3}\rho(\x/\delta_h)$ and $\delta_h=h^{1/5}$, so that the argument in Step 2 of the proof of Lemma \ref{conv} yields  $\bar\u_{h}\to\u_*$ in $H^1(\om,\mathbb R^3)$ along with $h\|\nabla\bar\u_{h}\|_{L^{\infty}(\om,\mathbb R^{3\times3})}\to 0$ as $h\to 0$, and moreover
\begin{equation*}
\lim_{h\to0}\frac1{h^2}\int_\om\mathcal W(\x,\mathbf I+h\nabla\bar\u_h)\,d\x=\frac12\int_\om\mathbb E\u_*\,D^2\mathcal W(x,\mathbf I)\,\mathbb E\u_*\,d\x.
\end{equation*}
We let $\u_h^*:=\bar\u_h-\tfrac12\mathbf w_h\wedge\x$, where $\w_h:=|\om|^{-1}\int_\om\mathrm{curl}\,\bar\u_h$, and $\mathbf w:=|\om|^{-1}\int_\om\mathrm{curl}\,\u_*$, so that also along $\u_h^*$ we have
\begin{equation*}
\lim_{h\to0}\frac1{h^2}\int_\om\mathcal W(\x,\mathbf I+h\nabla\u_h^*)\,d\x=\frac12\int_\om\mathbb E\u_*\,D^2\mathcal W(\x,\mathbf I)\,\mathbb E\u_*\,d\x,
\end{equation*}
since the linear elastic energy is unaffected by the addition of infinitesimal rigid displacements, and similarly we have $\mathcal L(\u_h^*)\to\mathcal L(\u_*)$ as $h	\to0$ thanks to \eqref{L1}.
Moreover, since $\u_h^*\to\u_*-\tfrac12\w\wedge\x$ in $H^1(\om,\mathbb R^3)$ as $h\to 0$, and since $\int_{\om}|\mathbb E\u_*|^2\le 1/4$, we see that for small enough $h$ there holds $\int_\om |h\mathbb E\u_h^*|^2=\int_\om |h\mathbb E\bar\u_h|^2\le h^2$, so that $h\u^*_h$ is admissible for problem \eqref{minh}. 

\textbf{Step 2.}
For every $h\in(0,1)$, let $\v_h $ as in \eqref{minh} and $\u_h:=h^{-1}\v_h$.  By Step 1 we get
\begin{equation}\label{uno}
\mathcal F(\v_h;h\mathcal L)\le \mathcal F(h\u^*_h;h\mathcal L)= \frac{h^2}2\int_\om\mathbb E\u_*\,D^2\mathcal W(x,\mathbf I)\,\mathbb E\u_*\,d\x
-h^2\mathcal L(\u_*)+o(h^2)
\end{equation}
as $h\to 0$.
On the other hand, 
by  \eqref{rigidity}  and thanks to the Euler-Rodrigues formula yielding the representation $$\mathbf R_h=\mathbf I+\sin\theta_h\mathbf W_h+(1-\cos\theta_h)\mathbf W^2_h$$
for some suitable $\theta_h\in\mathbb R$ and $\mathbf W_h\in\mathbb R^{3\times3}_{\mathrm{skew}}$ with $|\mathbf W_h|^2=|\mathbf W_h^2|^2=2$,
 we have
 \begin{equation}
 \begin{aligned}
&\displaystyle\frac{1}{C_{\om}}\int_\om \left(\left|\mathbb E\v_h-(1-\cos\theta_h)\mathbf W^2_h\right|^2+\left |\mathrm{skew} (\nabla\v_h)-\sin\theta_h\mathbf W_h\right |^2\right)d\x \\&\qquad=
\displaystyle\frac1{C_{\om}}\int_{\om}|\mathbf I-\mathbf R_h+\nabla\v_h|^2\,d\x\le \int_\om \mathcal W(\x,\mathbf I+\nabla\v_h)\,d\x=h\mathcal L(\v_h)+\mathcal F(\v_h;h\mathcal L).
\end{aligned}
\end{equation}
By taking into account that $\int_\om \curl \v_h=\mathbf 0$, i.e., $\int_\om \mathrm{skew} (\nabla\v_h)=\mathbf 0$, an application of Jensen inequality yields
\begin{equation}\label{due}
\frac{1}{C_{\om}}\int_\om \left|\mathbb E\v_h-(1-\cos\theta_h)\mathbf W^2_h\right|^2\, d\x+\frac{|\om|}{C_\om}\left |\sin\theta_h\mathbf W_h\right |^2\le h\mathcal L(\v_h)+\mathcal F(\v_h;h\mathcal L).
\end{equation}
From \eqref{minh}, \eqref{uno} and \eqref{due}, by \eqref{KP} and by the energy identity \eqref{energyidentity}  we deduce %
\begin{equation}\label{cinque}\begin{aligned}
&\frac{1}{C_{\om}}\int_\om \left|\mathbb E\v_h-(1-\cos\theta_h)\mathbf W^2_h\right|^2\, d\x+\frac{|\om|}{C_\om}\left |\sin\theta_h\mathbf W_h\right |^2\\&\qquad\le
-\frac{h^2}2\int_\om\mathbb E\u_*\,D^2\mathcal W(\x,\mathbf I)\,\mathbb E\u_*\,d\x+K_\Omega h\|\mathcal L\|_{ *}\,\|\mathbb E\v_h\|_{L^2(\om,\mathbb R^{3\times3})}+o(h^2)\\
&\qquad\le K_\Omega h^2\|\mathcal L\|_{ *}+o(h^2)
\end{aligned}
\end{equation}
as $h\to 0$.
By setting $$\mathbf A_h:=h^{-1}(1-\cos\theta_h)\mathbf W^2_h\quad\mbox{and}\quad \mathbf B_h:=h^{-1}\sin \theta_h\mathbf W_h,$$
 thanks to \eqref{cinque} we deduce that 
\beeq\lab{sei}
\begin{aligned}
&\int_\om \left|\mathbb E\u_h-\mathbf A_h\right|^2\, d\x+|\om||\mathbf B_h|^2\le
C_{\om}K_\Omega\|\mathcal L\|_{ *} +o(1)
\end{aligned}
\eneq
as $h\to 0$.
 In particular, since $\int_\Omega |\mathbb E\u_h|^2\le 1$, we deduce that $|\Omega||\mathbf A_h|^2\le 2+2C_{\om}K_\Omega\|\mathcal L\|_{ *}+o(1)$ as $h\to 0.$ Then we get \KKK
$|\om|
 |\mathbf B_h|^2 \le  C_{\om}K_\Omega\|\mathcal L\|_*+o(1)
$ as $h\to 0$, hence $\mathbf A_h\to 0$ as $h\to 0$. 
\textbf{Step 3.}
 We end the proof by contradiction, assuming  that there exists a vanishing sequence $(h_n)_{n\in\mathbb N}\subset(0,1)$ such that \begin{equation}\label{=1}h_n^2\int_\om |\mathbb E\u_{h_n}|^2=\int_\om |\mathbb E\v_{h_n}|^2=h_n^2\end{equation}
  for all $n\in\mathbb N$. Thanks to Lemma \ref{conv} \MMM(applied with $M=1$)\KKK,    along a not relabeled subsequence we have $\mathbb E\u_{h_n}\rightharpoonup\mathbb E\u_{*}$ weakly in $L^2(\om,\R^{3\times3})$
 and thanks to \eqref{=1} and to \eqref{sei} 
we have 
\begin{equation*}\begin{aligned}
1-2\int_\om \mathbb E\u_{h_n}:\mathbf A_{h_n}
&=\left(\int_\om |\mathbb E\u_{h_n}|^2-2\int_\om \mathbb E\u_{h_n}:\mathbf A_{h_n}\right)
\\&\le   \int_\om \left|\mathbb E\u_{h_n}-\mathbf A_{h_n}\right|^2\le C_\om K_\Omega\|\mathcal L\|_{ *} + o(1)\end{aligned}\end{equation*}
{as $n\to+\infty$}.
We pass to the limit as $n\to+\infty$ and by recalling that $\mathbf A_{h_n}\to 0$ we obtain
\begin{equation*}1\le C_{\om}K_\Omega\|\mathcal L\|_{ *},
\end{equation*}
which is
a contradiction with \eqref{h2}.
\end{proof}

\begin{corollary}\lab{locmin} Suppose that the assumptions of {\rm Theorem \ref{TH1}} are satisfied and
let \MMM $$M_0>\max\left\{1,\frac{\|\mathcal L\|_*}{\gamma(\om,C)}\right\},$$ where $\gamma(\om,C)$ is defined in {\rm Lemma \ref{intmin}}. There exist  $h_0>0$ such that if $0<h<h_0$ and $\v_h$ is a solution to
 \begin{equation}\label{mh}\min \left \{\mathcal F( \v; h\mathcal L): \v\in H^1(\om,\mathbb R^3),\ \int_\om |\mathbb E(\v)|^2\le M_0^2h^2,\,\int_\om\mathrm{curl}\,\v=\mathbf0\right \},
 \end{equation}
then $\v_h$ is a local minimizer for $\mathcal F (\cdot; h\mathcal L)$  over $H^1_{\mathbf I}(\om,\mathbb R^3)$ and $\u_h:=h^{-1}\v_h$ is a local minimizer for $\mathcal F_h (\cdot; \mathcal L)$  over $H^1_{\mathbf I}(\om,\mathbb R^3)$.
\end{corollary}
\begin{proof} 
The existence of a solution to problem \eqref{mh} is due to Lemma \ref{exist} and Remark \ref{minremark}.
\MMM
Since  $\|M_0^{-1}\mathcal L\|_*< \gamma(\om, C)$, \KKK we may apply Lemma \eqref{intmin} with  $M_0^{-1}\mathcal L$ in place of $\mathcal L$, thus finding $h_0>0$  such that $\v_h$ satisfies
$\int_\om |\mathbb E\v_h|^2< M_0^2h^2$  as soon as $h\in(0,h_0)$.

 Let $\boldsymbol \psi\in H^1(\om,\mathbb R^3)$ be such that $\int_\om\mathrm{curl}\,\boldsymbol\psi=\mathbf 0$. For any given $h\in (0, h_0)$, the continuity of the map $[0,1]\ni\eps\mapsto\int_\om |\mathbb E\v_h+\eps\boldsymbol\psi|^2$ shows that there exists $\eps_0=\eps_0(\v_h,\boldsymbol\psi)$ such that for every $\eps<\eps_0$ there holds $\int_\om |\mathbb E\v_h+\eps\boldsymbol\psi|^2<M_0^2h^2$. Thus for every $\eps<\eps_0$
we have that $\v_h+\eps\boldsymbol\psi$ is admissible for problem \eqref{mh} and so $\mathcal F(\v_h; h\mathcal L)\le \mathcal F(\v_h+\eps\boldsymbol\psi; h\mathcal L)$ as claimed. The  statement about $\u_h$ follows from Remark \ref{ftofh}.
\end{proof}

\begin{proofth1} 
Let $M_0$ and $h_0$ be as in Corollary \ref{locmin} and let $(h_j)_{j\in\mathbb N}\subset(0,h_0)$ be a vanishing sequence. For every $j\in\mathbb N$, let $\v_j$ be a solution to problem \eqref{mh} (with $h_j$ in place of $h$), whose existence is ensured by Remark \ref{minremark}.
By Corollary \ref{locmin}, for every $j\in\mathbb N$ we have that $\v_j$ is a local minimizer  for $\mathcal F (\cdot ; h_j\mathcal L)$ over $H^1_{\mathbf I}(\om,\R^{3\times 3})$, and $\u_j:=h_j^{-1}\v_j$ is a local minimizer  for $\mathcal F_{h_j}(\cdot;\mathcal L)$ over $H^1_{\mathbf I}(\om,\R^{3\times 3})$. Moreover, by Lemma \ref{conv} we get, up to subsequences,
$\mathbb E(\u_j)={h_j}^{-1}\mathbb E(\v_j)\wconv\mathbb E(\u_*)$
 weakly in  $L^2(\Omega,\R^{3\times 3})$ as $j\to\infty$  and
\[
\lim_{j\to \infty}\mathcal F_{h_{j}}(\u_j;\mathcal L)=\lim_{j\to +\infty}h_j^{-2}\mathcal F(\v_j; h_{j}\mathcal L)= \mathcal F_0(\u_*;\mathcal L),
\]
where $\u_*\in H^1(\om,\mathbb \R^3)$ is a solution to
$$\min_{\u\in H^1(\om,\mathbb \R^3)}\left\{\mathcal F_0(\u;\mathcal L): \int_\om |\mathbb E\u|^2\le M^2_0\right\}.$$
\MMM Since the choice of $M_0$ from Corollary \ref{locmin} is such that
$
M_0^{-1}\|\mathcal L\|_*< \gamma(\om, C) 
$, \KKK
we obtain 
 $\|\mathcal L\|_*\le \tfrac{CM_0}{K_\Omega}$, thus by Lemma \ref{basiclemma} we deduce that $\u_*$ minimizes $\mathcal F(\cdot;\mathcal L)$ over the whole $H^1(\om,\mathbb R^3)$. \end{proofth1}

\KKK
\begin{proofad5} Since $\mathbf R\in \mathcal S^0_{\mathcal L}$, then $\mathbf R^T\mathcal L$ satisfies \eqref{L1} as remarked in Section \ref{mainsection}. Hence, by Theorem \ref{TH1} there exist a sequence $\{h_j\}_{j\in\mathbb N}\subset(0,1)$ and local minimizers $\v_j$ for 
$\mathcal F(\cdot\,; h_j\mathbf R^T\mathcal L)$ \MMM over $H^1_{\mathbf I}(\om,\mathbb R^3)$ \KKK 
 such that 
\[ \lim_{j\to\infty}h_j^{-2}\mathcal F(\v_j\,; h_j\mathbf R^T\mathcal L)= \beta(\mathbf R),\]
where $\beta$ is defined by \eqref{betti}.
By setting $\mathbf y_j(\x):= \mathbf R \x+h_j\mathbf R\v_j$ we get $\int_\om \curl \mathbf R^T\mathbf y_j=0$. By taking  into account frame indifference and again that $\mathbf R\in \mathcal S^0_{\mathcal L}$ we get 
\[\begin{aligned} \mathcal G(\mathbf y_j\,; h_j\mathcal L)&=\int_\om\mathcal W(\x,\mathbf R+h_j\nabla\mathbf R\v_j)\,d\x-h_j\mathcal L(\mathbf y_j-\x)\\
&\displaystyle=\int_\om\mathcal W(\x,\mathbf I+h_j\nabla\v_j)\,d\x-h_j\mathcal L(\mathbf y_j-\mathbf R \x)\\
&\displaystyle =\int_\om\mathcal W(\x,\mathbf I+h_j\nabla\v_j)\,d\x-h_j^2\mathcal L(\mathbf R \v_j)= \mathcal F(\v_j\,;h_j\mathbf R^T\mathcal L).
\end{aligned}\]
Therefore $\mathbf y_j$ is a  local minimizer for $\mathcal G(\cdot\,; h_j\mathcal L)$ \MMM over $H^1_{\mathbf R}(\om,\mathbb R^3)$  \KKK
and 
\[\lim_{j\to\infty}h_j^{-2}\mathcal G(\mathbf y_j\,; h_j\mathcal L)= \lim_{j\to\infty}h_j^{-2}\mathcal F(\v_j\,; h_j\mathbf R^T\mathcal L)= \beta(\mathbf R)\]
as claimed.
\end{proofad5}

\begin{proofad6} Since $\mathcal S^0_{\mathcal L}\setminus\{\mathbf I\}\neq \emptyset$ and since $\beta$ is not a constant map then $\beta(\mathcal S^0_{\mathcal L})=[a,b]$ for some couple of reals $a,b$, $a< b$, and so for every $n\in \mathbb N$ there exists $\{\mathbf R_k:\, k=1,2...n\}\subset \mathcal S^0_{\mathcal L}$ such that $\beta(\mathbf R_k)\not=\beta(\mathbf R_k')$ if $k\not=k'$. By Corollary \ref{CR1} for every vanishing sequence $(h_j)_{j\in\mathbb N}\subset(0,1)$ and for every $k=1,2....n$ there exist a subsequence $(h_j^{(k)})_{j\in\mathbb N}$ and  local minimizers $\mathbf y_j^{(k)}$  for $\mathcal G(\cdot\,; h_j^{(k)}\mathcal L)$ \MMM over $H^1_{\mathbf R_k}(\om,\mathbb R^3)$ \KKK 
such that
\[\lim_{j\to\infty}\left({h_j^{(k)}}\right)^{-2}\mathcal G(\mathbf y_j^{(k)}\,; h_j^{(k)}\mathcal L)= \beta(\mathbf R_k).\]
We let
\[\delta_n:=\min\{ |\beta(\mathbf R_k)-\beta(\mathbf R_k')|: k\not=k', \ k,k'=1,2,...n\}.\]
It is readily seen that for every $k=1,2....n$ there exists \MMM $h_0^{(k)}> 0$ \KKK such that, for every $0< h< h_0^{(k)}$, the functional $\mathcal G(\cdot\,; h\mathcal L)$ has a local minimizer $\mathbf y^{(k)}$ \MMM over $H^1_{\mathbf R_k}(\om,\mathbb R^3)$ \KKK satisfying 
$$|\mathcal G(\mathbf y^{(k)}\,; h\mathcal L)- \beta(\mathbf R_k)| < \delta_n/2.$$ The result follows by choosing $h_0:=\min \{h_0^{(k)}: k=1,2...n\}$.
\end{proofad6}

\section{Constrained local minimizers at different energy levels: an example}\label{sectionexample}
In this Section we give an explicit example in which the situation described in Corollary \ref{CR1} and Corollary \ref{CR2}  occurs. To this aim we shall consider  the  Yeoh type energy density defined by \eqref{iso-vol} and \eqref{Yeoh}, with $\mathcal W_{\textup{vol}}(\mathbf F)=g(\det\mathbf F)$, where
$g:\mathbb R_+\to\mathbb R$ is the convex $C^2$  function (satisfying \eqref{guno} and \eqref{phi_growth} with $r=2$) obtained by setting \begin{equation*}g(t)=c(t^2-1-2\log t)\end{equation*} for some $c>0$. This is a usual choice, see \cite{H}.

By taking into account that  for every $\mathbf B\in \mathbb R^{3\times3}_{\mathrm{sym}}$ there holds
\begin{equation*}\lab{taylor}
\frac{|\mathbf I+ \eps \mathbf B|^2}{\det (\mathbf I+\eps \mathbf B)^{2/3}}-3= \eps^2 (2|\mathbf B|^2-\tfrac{4}{3}|\tr \mathbf B|^2)+o(\eps^2)
\end{equation*}
as $\eps\to0$,
and  since 
\begin{equation*}
\mathcal W_{\textup{vol}}(\mathbf I+ \eps \mathbf B)=g(\det(\mathbf I+ \eps \mathbf B))=\frac{\eps^2}{2}g ''(1)|\tr \mathbf B|^2+ o(\eps^2)=2c|\tr \mathbf B|^2+ o(\eps^2),
\end{equation*}
 we get 
 \begin{equation*}\lab{linstrain}
\frac{1}{2}\mathbf B\,D^2\mathcal W(\mathbf I)\,\mathbf B=2c_1|\mathbf B|^2+ ( 2c-\tfrac{4}{3}c_1)|\tr \mathbf B|^2
\end{equation*}
and  we choose from now on  $c_1=2,\ c=\tfrac{4}{3}$ so that $$ \frac{1}{2}\mathbf B\,D^2\mathcal W(\mathbf I)\,\mathbf B=4|\mathbf B|^2.$$
 
 Let now
\begin{equation}\label{cyl2}\Omega:=\{\x=(x,y,z)\in\mathbb R^3: x^2+y^2<1,\,0<z<1\},\end{equation} 
 $B:=\{(x,y)\in\mathbb R^2: x^2+y^2<1\},$  let $\phi:B\to\R$ be the radial function whose radial profile (still denoted by $\phi$) is
 \beeq\lab{phi}\phi(r):= \log r+r^2-3r+2,\qquad r:=\sqrt{x^2+y^2},\eneq 
 and let us consider a load functional of the form
  \beeq\lab{f} \mathcal L(\u)=\int_\Omega \mathbf f\cdot \u\,d\x\qquad\mbox{where}\qquad\mathbf f(\x)=\mathbf f(x,y,z):=r^{-1}\phi'(r)({x},{y}, 0).\eneq
  It is readily seen that $\phi(1)=\phi'(1)=0$ and that $\mathbf f\in L^p(\om,\R^3)$ if and only if  $p< 2$ (\MMM so that assumption \eqref{intforces} is not satisfied\KKK) and thus $\mathcal L\in(H^1(\om,\R^3))^*$. \MMM It is also easy to see that $\mathcal L$ is equilibrated. We claim that $\mathcal L((\mathbf R-\mathbf I)\x)=0$
 for every $\mathbf R\in SO(3)$, that is, condition \eqref{L1plus} is satisfied and $\mathcal S^0_\mathcal L\equiv SO(3)$ (\MMM which implies that every direction in $\mathbb R^3$ is an axis of equilibrium of $\mathcal L$, because the astatic load is zero\KKK). Indeed if $\mathbf R\in SO(3)$ then by the Euler-Rodrigues formula there exist $\theta\in [0,2\pi)$
 and $\mathbf a\in \mathbb R^3,\ |\mathbf a|=1$, such that for every $\x\in\om$
 \[\mathbf R\x=\x+\sin\theta\, (\mathbf a\wedge \x)+(1-\cos\theta)\,\mathbf a\wedge (\mathbf a\wedge \x),\]
 hence by setting $c(\mathbf a,\theta)=(1-\cos\theta)(\pi(1-a_3^2)-1)$ we have
 \[\begin{aligned}\displaystyle\int_\om \mathbf f(\x)\cdot (\mathbf R-\mathbf I)\x\,d\x&=\sin\theta\int_\om \mathbf f(\x)\cdot (\mathbf a\wedge \x)\,d\x+(1-\cos\theta)\int_\om \mathbf f(\x)\cdot((\mathbf a\cdot \x)\mathbf a-\x)\,d\x\\
 &=\displaystyle c(\mathbf a,\theta)\int_0^1 r^2\phi'(r)\,dr= -2c(\mathbf a,\theta)\int_0^1 r\phi(r)\,dr =0
 \end{aligned}\]
 as claimed. 
 
 Let us now consider the following rotation of $\pi/2$ around the $z$ axis \begin{equation*}\label{tildematrix}
{\mathbf R}_*:={\footnotesize\left(\begin{array}{ccc}0&-1&0\\1&0&0\\0&0&1\end{array}\right)}.
 \end{equation*}
 We claim that $\beta({\mathbf R}_*)< \beta(\mathbf I)$, that is, the map $\beta$ defined in \eqref{betti} takes two different values as required in Corollary \ref{CR2}. To this aim we need the following
 \begin{lemma}\label{lemma21} Let $\om$ as in \eqref{cyl2},
 $\phi$ as in \eqref{phi} and $\mathcal L$ as in \eqref{f}.  Then 
\[
\min_{H^1(\Omega;\mathbb R^3)}{\mathcal F}_0(\cdot\,;{\mathbf R}_*^T\mathcal L)\le \min_{u\in H^2(B)} \int_{B}8u_{xy}^2+2(u_{yy}-u_{xx})^2-\int_B (u_{y}\phi_{y}+u_{x}\phi_{x})< 0. \]
 \end{lemma}
 \begin{proof}
Let $B$  the unit ball in $\mathbb R^2$ and $$\mathcal K:=\{(u_{y},-u_{x}, 0): u\in H^2(B)\}\subset H^1(\Omega;\mathbb R^3).$$  It is easily seen that $\mathcal K$ is weakly closed in $H^1(\Omega,\mathbb R^3)$ so  there are minimizers of ${\mathcal F}_0(\cdot\,;{\mathbf R}_*^T\mathcal L)$ over $\mathcal K$ (recalling that $\mathbf R_*^T\mathcal L$ satisfies \eqref{L1} since the $z$ axis is an axis of equilibrium for $\mathcal L$).
 Therefore
  \begin{equation}\label{comparison}
\displaystyle  \min_{\v\in H^1(\Omega;\mathbb R^3)}{\mathcal F}_0(\v\,;{\mathbf R}_*^T\mathcal L)\le  \min_{u\in H^2(B)} \int_{B}8u_{xy}^2+2(u_{yy}-u_{xx})^2-\int_B (u_{y}\phi_{y}+u_{x}\phi_{x})
\end{equation}
If $u$ is a minimizer of the right hand side of \eqref{comparison}, $0< \eps< 1$, $B_\eps:=\{(x,y)\in\mathbb R^2:x^2+y^2<\eps\}$ and $\zeta\in H^2_0(B\setminus B_\eps)$, by taking into account that $\Delta\phi\in L^2(B\setminus B_\eps)$, a first variation argument for the minimization problem in the right hand side of \eqref{comparison} yields 
 \begin{equation*}\label{variation}\begin{aligned}
  4\int_{B\setminus B_\eps} (4u_{xy}\zeta_{xy}+u_{yy}\zeta_{yy}+u_{xx}\zeta_{xx}-u_{yy}\zeta_{xx}-u_{yy}\zeta_{xx})&
 =\int_{B\setminus B_\eps} (\zeta_{y}\phi_{y}+\zeta_{x}\phi_{x})= -\int_{B\setminus B_\eps}\zeta\Delta\phi
 \end{aligned}
 \end{equation*}
 and after integration by parts we obtain the Euler-Lagrange equation $$4\Delta^2 u+\Delta\phi=0\qquad\mbox{ in $\mathcal D'(B\setminus B_\eps)$},$$
 for every $0< \eps < 1$, where $\Delta^2$ denotes the planar biharmonic operator. On the other hand if $u$ is a minimizer then by using $u$ as test function  we easily get
 \beeq  \lab{eulermin}\displaystyle\int_{B}16u_{xy}^2+4(u_{yy}-u_{xx})^2+(u_{y}\phi_{y}+u_{x}\phi_{x})=0.
  \eneq
In view of \eqref{comparison}, and since ${\mathcal F}_0(\mathbf 0\,;{\mathbf R}_*^T\mathcal L)=0$, in order to conclude it is enough to show that $\min_{\v\in \mathcal K}{\mathcal F}_0(\v\,;{\mathbf R}_*^T\mathcal L)\neq 0$.
 Assume by contradiction 
that $\min_{\v\in \mathcal K}{\mathcal F}_0(\v\,;{\mathbf R}_*^T\mathcal L)= 0$: then by \eqref{eulermin} we get
  \[ \displaystyle\int_{B}4u_{xy}^2+(u_{yy}-u_{xx})^2=0
  \]
  that is $u_{xy}=u_{yy}-u_{xx}=0$ a.e. in $B$ hence $\Delta^2 u=0$ a.e.
   in $B$ and from the  above  Euler-Lagrange equation we deduce $\Delta\phi\equiv 0$ in $B\setminus B_\eps $,  a contradiction since $\Delta\phi = 3-\frac{2}{r}$  in $B\setminus B_\eps $.
  \end{proof}
  
  We can now conclude by proving the claim, i.e., by proving that $\beta({\mathbf R}_*)< \beta(\mathbf I)$.
  Let  $\v=(v_1,v_2,v_3)\in H^1(\Omega,\mathbb R^3)$,  $\tilde \v:=(v_1,v_2)$ and let $\tilde \u\in H^1(B,\mathbb R^2)$ be defined by $\tilde\u(x_1,x_2):=\int_0^1 \tilde\v(x_1,x_2,x_3)\,dx_3$. By Jensen inequality it is readily seen that
  \[ \displaystyle\mathcal F_0(\v\,;\mathcal L) \ge 4\int_B|\widetilde{\mathbb E}(\tilde \u)|^2-\int_B\nabla\phi\cdot\tilde \u
  \]
  where $\widetilde{\mathbb E}(\cdot)$ is the upper-left $2\times2$ submatrix of $\mathbb E(\cdot)$. By arguing as in the proof of Theorem 2.7 of \cite{MP}, if
  \begin{equation*}\label{etastar}
\eta_*(r)=-\frac1{16}\,r\phi(r)+\frac1{16r}\int_0^r t^2\phi'(t)\,dt
\end{equation*}
then  the radial function defined by $w(r):=\int_0^r\eta_*(t)\,dt$ belongs to $H^2(B)$ and $\nabla w$ minimizes 
$$\mathcal J(\tilde\u):=4\int_B|\widetilde{\mathbb E}(\tilde \u)|^2-\int_B\nabla\phi\cdot\tilde \u$$
over $H^1(B)$, hence  $w$ minimizes
$$4\int_B |D^2 v|^2-\int_B\nabla\phi\cdot\nabla v$$
among all $v$ in  $H^2(B)$, where $D^2$ denotes the Hessian in the $x,y$ variables. Therefore for every $0 < \eps <1$ and for every $\psi\in H^2_0(B\setminus B_\eps)$
$$8\int_{B\setminus B_\eps} D^2 v\cdot D^2 \psi-\int_{B\setminus B_\eps}\nabla\phi\cdot\nabla \psi=0,$$
that is, $w$ solves the biharmonic equation
$8\Delta^2 w=-\Delta\phi$ in $B\setminus B_\eps$ and since $\Delta\phi$ is not identically zero then $\Delta w$ is not identically zero as well.
This implies by Young inequality 
\begin{equation}\label{young}
\int_B 8w_{xy}^2+2(w_{yy}-w_{xx})^2-\int_B\nabla w\cdot\nabla\phi<\int_B  8w_{xy}^2+4w_{yy}^2+4w_{xx}^2-\int_B\nabla w\cdot\nabla\phi.
\end{equation}
Notice that the inequality is strict, since the Young inequality $2(w_{yy}-w_{xx})^2\le 4w_{xx}^2+4w_{yy}^2$ holds with equality if and only if $w_{xx}=-w_{yy}$, and we have just checked that $\Delta w$ does not vanish identically on $B$. In particular, $2(w_{xx}-w_{yy})^2< 4w_{xx}^2+4w_{yy}^2$ on a set of positive measure in $B$.
 From Lemma \ref{lemma21} and from \eqref{young} we infer 
\[\begin{aligned}&\beta(\mathbf R_*)=\min_{\u\in H^1(\Omega,\mathbb R^3)}{\mathcal F}(\u\,;\mathbf R_*^T\mathcal L)
\le \min_{v \in H^2(B)} \int_B 8v_{xy}^2+2(v_{yy}-v_{xx})^2-\int_B\nabla v\cdot\nabla\phi\\&
\le \int_B 8w_{xy}^2+2(w_{yy}-w_{xx})^2-\int_B\nabla w\cdot\nabla\phi< \int_B  8w_{xy}^2+4w_{xx}^2+4w_{yy}^2-\int_B\nabla w\cdot\nabla\phi\\&
= \min_{v\in H^2(B)}\int_B  8v_{xy}^2+4v_{xx}^2+4v_{yy}^2-\int_B\nabla v\cdot\nabla\phi\le \min_{\u\in H^1(\Omega,\mathbb R^3)}{\mathcal F}(\u\,;\mathcal L)=\beta(\mathbf I)\end{aligned}\]
 as claimed.
 In particular, Corollary \ref{CR2} applies to this example.

\subsection*{Acknowledgements} 
 The authors wish to warmly thank professor Giuseppe Puglisi for his kind suggestions. \KKK

The authors acknowledge support from the MIUR-PRIN  project  No 2017TEXA3H.
The authors are members of the
GNAMPA group of the Istituto Nazionale di Alta Matematica (INdAM).

\end{document}